\documentclass[12pt]{amsart}

\usepackage{amssymb}

\usepackage{enumitem}

\usepackage{graphicx}

\makeatletter
\@namedef{subjclassname@2020}{%
  \textup{2020} Mathematics Subject Classification}
\makeatother

\usepackage[T1]{fontenc}

\newtheorem{theorem}{Theorem}[section]

\newtheorem{lemma}[theorem]{Lemma}

\theoremstyle{definition}

\newtheorem{remark}[theorem]{Remark}

\numberwithin{equation}{section}

\theoremstyle{plain}
\newtheorem{thm}{Theorem}[section]
\newtheorem*{starthma}{Theorem A}

\newtheorem*{starthm}{Theorem}

\newtheorem*{starthm1}{Main Theorem 1}
\newtheorem*{starthm2}{Main Theorem 2}
\newtheorem*{starthm3}{Main Theorem 3}

\newtheorem{prop}[thm]{Proposition}
\newtheorem{defn}{Definition}
\input txdtools

\newcommand{\cali}{{\mathcal I}}

\newcommand{\calg}{{\mathcal G}}

\newcommand{\call}{{\mathcal L}}

\newcommand{\calp}{{\mathcal P}}

\newcommand{\calr}{{\mathcal R}}
\newcommand{\cals}{{\mathcal S}}

\newcommand{\calu}{{\mathcal U}}

\newcommand{\calz}{{\mathcal Z}}

\newcommand{\CC}{{\mathbb C}}

\newcommand{\NN}{{\mathbb N}}

\renewcommand{\hat}{\widehat}

\newcommand{\la}{\lambda}

\begin{document}

\title[Non-Ergodicity]{Meromorphic functions whose action on  their Julia sets is  Non-Ergodic }

\author{Tao Chen, Yunping Jiang and Linda Keen}

\address{Tao Chen, Department of Mathematics, Engineering and Computer Science,
Laguardia Community College, CUNY,
31-10 Thomson Ave. Long Island City, NY 11101 and
CUNY Graduate Center, New York, NY 10016
}
\email{tchen@lagcc.cuny.edu}

\address{Yunping Jiang, Department of Mathematics, Queens College of CUNY,
Flushing, NY 11367 and Department of Mathematics and the CUNY Graduate Center, New York, NY 10016}
\email{yunping.jiang@qc.cuny.edu}

\address{Linda Keen, Department of Mathematics, the CUNY Graduate
School, New York, NY 10016}
\email{LKeen@gc.cuny.edu; linda.keenbrezin@gmail.com}



\begin{abstract}
 Nevanlinna functions are meromorphic functions with a finite number of asymptotic values and no critical values.  In \cite{KK2} it was proved that if the orbits of all the asymptotic values accumulate on a compact set on which the function acts as a repeller, then the function acts ergodically on its Julia set.  In \cite{CJK4}, we proved the action of the function on its Julia set is still ergodic  if some, but not all of the asymptotic values land on infinity, and the remaining ones land on a compact repeller.   In this paper, we complete the characterization of ergodicity for Nevanlinna functions by proving that if all the asymptotic values land on infinity,
   then the Julia set is the whole sphere and the action of the  map there  is non-ergodic.
\end{abstract}

\maketitle

\section{Introduction}

 It is a theorem of McMullen \cite{McM} and Lyubich \cite{Lyu1}) that if $f$ is a rational map of degree greater than one, and if $P(f)$ is its post-singular set (see the formal definition in the next section), one of two things holds: either the Julia set $J(f)$ is equal to the whole Riemann sphere and the action of $f$ is ergodic or,  for almost every $z$ in $J(f)$, the spherical distance $d(f^n(z),P(f)) \rightarrow 0$ as $n \to \infty$; that is, the $\omega$-limit set $\omega(z)$ is a subset of $P(f)$ that varies with $z$.

 In the same vein, Bock~\cite{Boc} proved  a  dichotomy theorem  for  transcendental functions:  either  the Julia set is the whole sphere and for any set  $A$ of positive Lebesgue measure, the orbits of almost all points in $\CC$ have infinitely many iterates that land in $A$, or  the Julia set is not the whole sphere, and almost every point in the Julia set is attracted to the post-singular set. Unlike the rational case, though, when the Julia set is the whole sphere, it is not known when such functions are ergodic.  Misiurewicz~\cite{Mis} has shown that when $f(z)=e^z$, the Julia set is the whole sphere, but Lyubich [Lyu2], following earlier work \cite{GGS,Dev}, has shown that its action is \emph{not} ergodic.   More examples of entire maps that \emph{are} ergodic can be found in \cite{WZ,CW}.  Skorulski, ~\cite{Skor,Skor1},  answered  the ergodicity question for  meromorphic maps of the form $f(z)=\la \tan z$ and $f(z)=\frac{a e^{z^p}+b e^{-z^p}}{c e^{z^p} + d e^{z^{-p}}}$ that have two asymptotic values.

 A more general sort of affirmative result can be found in  \cite{KK2}:  If the orbits of the singular values either land on or accumulate on a compact repelling set, and if the  map $f$ satisfies an additional growth rate condition,  then its Julia set is the whole sphere and it is ergodic.

 In this paper, we study the ergodicity question for meromorphic functions whose singular set consists of  finitely many asymptotic values and no critical values, the so-called {\em Nevanlinna functions}.   The dynamical properties of these functions were first investigated in \cite{DK}, where it was shown that they share many of the  dynamical properties of rational maps. In~\cite{CJK4}, the authors studied the ergodicity question for subfamilies of Nevanlinna functions for which some of the  asymptotic values have orbits that land on infinity.  They proved:

\begin{starthma}
Let $f$ be a Nevanlinna function.   If some, but not all, of the asymptotic values of $f$ have orbits that land on infinity, and if the remaining asymptotic values have  $\omega$-limit sets that  are compact repellers, then the Julia set of $f$  is the Riemann sphere, and $f$ is  ergodic there.
\end{starthma}

In this paper we complete the characterization of ergodicity for Nevanlinna functions by proving

\begin{starthm1}
 Let $f$ be a Nevanlinna function with $N$ asymptotic values, $\lambda_1, \ldots, \lambda_N$. 
 If  all $N$  asymptotic values land on infinity,  then  the Julia set of $f$ is the Riemann sphere, and the action of $f$ is \emph{not} ergodic there.
 \end{starthm1}
 
 In addition, under the same hypotheses for the asymptotic values of $f$,  we prove the following theorems.

 \begin{starthm2}
 For almost every point $z\in \mathbb{C}$, its $\omega$-limit set satisfies:
 $$\omega(z)=P_f=\cup_{i=1}^N\{\lambda_i, f(\lambda_i),\ldots, f^{p_i-1}(\lambda_i), \infty\}.$$
 \end{starthm2}

  \begin{starthm3}
 There is no $f$-invariant finite measure   absolutely continuous with respect to Lebesgue measure.
\end{starthm3}

Theorem A , together with Main Theorems 1, 2, and 3,   give a full answer to  ergodicity questions for Nevanlinna functions whose Julia set is the whole sphere and whose asymptotic values 
 land either on infinity or on a compact repeller. The questions of whether the conclusions of all three theorems hold for Nevanlinna functions with asymptotic values   whose orbits are infinite and whose accumulation sets are unbounded are still open.   Another open question  is whether   there exists a unique  $\sigma$-finite $f$-invariant measure absolutely continuous with respect to Lebesgue measure if all the asymptotic values of $f$  have finite orbits.

\medskip
The paper is organized as follows.  The first part of Section~\ref{prelim} contains basic definitions and standard theorems.  The remainder of that section is devoted to a detailed review of the properties of  Nevanlinna functions.    In particular, the asymptotic values divide a neighborhood of infinity into asymptotic tracts on which the function is exponentially contracting.   In order to prove the function is non-ergodic, we need to produce disjoint invariant wandering sets in these tracts,  each of which has positive Lebesgue measure.  This involves estimating the contraction on the asymptotic tracts.  In order to make these estimates, we use some classical techniques.  We introduce an ``auxiliary'' variable in each tract which converts it to a half plane in which we can make the estimates. 
Section~\ref{wanderingsets} contains the heart of the proof of  Main Theorem 1.   Distinct invariant wandering sets are constructed and  their Lebesgue measure is shown to be positive which proves the function is not ergodic.  The final section contains the proofs of Main Theorems 2 and 3.

\section{Preliminaries}\label{prelim}																															
\subsection{Definitions and standard theorems}\label{sec2.1}
A meromorphic function, $f: \CC \rightarrow \hat{\CC}$   is a local homeomorphism everywhere except at the set $S_f$ of  singular values.   In this paper, we will focus on functions whose singular set  is finite so that the singular values are isolated. We will assume this throughout the paper.   For these functions, the singular values  fall into two categories:  \\
Let $v\in \widehat{\mathbb{C}}$ be a singular value and let $V$ be a neighborhood of $v$.  Then
\begin{itemize}
\item If, for some component $U$ of $f^{-1}(V)$,  there is a  $u\in U$ such that $f'(u)=0$, then $u$ is   a {\em critical point} and $v=f(u)\in V$ is the corresponding  {\em critical value}, or
\item	 If, for some component $U$ of $f^{-1}(V)$, $f:U \rightarrow V \setminus \{ v \}$ is a universal covering map then $v$ is a {\em logarithmic asymptotic value}.   The component $U$ is called an {\em asymptotic tract} for $v$.   Any path $\gamma(t) \in U$ such that $\lim_{t \to 1} \gamma(t) = \infty$, $\lim_{t \to 1} f(\gamma(t))=v$ is called an {\em asymptotic path} for $v$.
\end{itemize}

 Note that the definition of an asymptotic tract depends on the choice of the neighborhood $V$.  If $V_1, V_2$ are punctured neighborhoods of $v$  and  $U_1$ and $ U_2$ are  unbounded components of their preimages such that $U=U_1 \cap U_2 \neq \emptyset$,  we call them {\em equivalent asymptotic tracts}.  For readability, we will not distinguish  between ``an asymptotic tract'' and its equivalence class.

 In the proofs of our results  we will repeatedly use the Koebe distortion theorems.  Many proofs exist in the standard literature on conformal mappings.  (See e.g. \cite[Theorem 6.16]{Z}),  For the reader's convenience, we state the theorems as we use them without proof.

\begin{thm}[Koebe Distortion Theorem]
Let $f: D(z_0, r)\to \mathbb{C}$ be a univalent function, then for any $\eta<1$,
\begin{enumerate}
\item $\displaystyle |f'(z_0)|\frac{\eta r}{(1+\eta)^2}\leq |f(z)-f(z_0)|\leq |f'(z_0)|\frac{\eta r}{(1-\eta)^2}$, $z\in D(z_0, \eta r)$,
\item If $\displaystyle T(\eta)=\frac{(1+\eta)^4}{(1-\eta)^4}$, $\displaystyle \frac{|f'(z)|}{|f'(w)|}\leq T(\eta)$, for any $z, w\in D(z_0, \eta r)$.
\item  $\displaystyle |\arg\frac{f'(z)}{f'(z_0)}|\leq 2\ln |\frac{1+\eta}{1-\eta}|$, for any $z\in D(z_0, \eta r)$.
\end{enumerate}
\end{thm}
\begin{thm}
Let $f: D(z_0, r)\to \mathbb{C}$ be a univalent function, and $\eta<1$.  Then, for any $A, B\subset D(z_0, \eta r)$,  
$$\frac{(1-\eta)^4}{(1+\eta)^4}\frac{m(A)}{m(B)}\leq \frac{m(f(A))}{m(f(B))}\leq \frac{(1+\eta)^4}{(1-\eta)^4}\frac{m(A)}{m(B)}.$$
\end{thm}

\subsection{Nevanlinna functions}
An important tool in studying meromorphic functions with finitely many critical points and finitely many asymptotic values is that they can be characterized by their Schwarzian derivatives.

\begin{defn} If $f(z)$ is a meromorphic function, its {\em Schwarzian derivative} is
$$
S(f) = (\frac{f''}{f'})' - \frac{1}{2}(\frac{f''}{f'})^2.
$$
\end{defn}
The Schwarzian differential operator satisfies the  chain rule condition,

$$ S(f\circ g) =S(f) g'^2 +S(g),$$
from which it is easy to deduce that if $f$ is a M\"obius transformation, $S(f)=0$, so that
$f\circ g$ and $g$ have the same Schwarzian derivative.

In \cite{N}, Chap. XI, \S 3, Nevanlinna,  shows how,  given a finite set of points in the plane and finite or infinite branching data for these points, this data defines, up to post-composition by a M\"obius transformation, a meromorphic function whose topological covering properties are determined by this data.   He proves,

 \medskip
\begin{thm}  The Schwarzian derivative of a meromorphic function with finitely many critical points and finitely many asymptotic values is a rational function.  If there are no critical points, it is a polynomial.  Conversely, if a meromorphic function has a rational Schwarzian derivative, it has finitely many critical points and finitely many asymptotic values. If the Schwarzian derivative is a polynomial of degree $m$, then the meromorphic function has $m+2$ asymptotic values and no critical points.
\end{thm}

In the literature, meromorphic functions with polynomial Schwarzian are often called {\em Nevanlinna functions}  (See e.g. \cite{C,EM}.)

\medskip
 In this paper, we continue our study of the  properties of the dynamical systems these functions generate.  Given a Nevanlinna function, we can define the {\em orbit} $\{ f^n{z} \}$ for any point $z \in \CC$.   In general, these orbits are infinite, but if, for some $m \geq 0$, $f^m(z)=\infty$, the orbit is finite.  Such points are called {\em prepoles of order $m$}.

 The standard defintions of stable (Fatou) and unstable (Julia) behavior for rational maps carry over to points with infinite orbits; the prepoles are unstable.  In \cite{DK}, it is proved that the classification of stable behavior for Nevanlinna maps is the same as that for rational maps.  In particular, if all the singular points are unstable, the Julia set is the whole sphere.

The results in \cite{KK2} and  \cite{CJK4} together show
\begin{starthm} If $f$ is a Nevanlinna function with $N$ asymptotic values $\la_1, \ldots, \la_N$, and if for some $0 \leq k<N$, $\la_i, i=1, \ldots, k$ are prepoles,\footnote{If some $\la_i=\infty$, it is a ''prepole of order 0".}  and if the $\omega$-limit sets of the remaining $N-k$  asymptotic values are compact repellers,  then the Julia set is the Riemann sphere and  $f$ is ergodic there. 
\end{starthm}

In this paper we complete the answer question of the ergodicity of the action of  Nevanlinna functions whose Julia set is the whole sphere by proving
\begin{starthm1}
 If $f$ is a Nevanlinna function with $N$ asymptotic values $\la_1, \ldots, \la_N$, and if ALL of them are prepoles, then  its Julia set is the Riemann sphere $\hat\CC$ and $f$ is not ergodic there. 
 \end{starthm1}

Before getting into the details we sketch the main  ideas in the proof (see also \cite{Lyu1,Skor}).   We construct two disjoint sets $A$ and $B$ such that:
\begin{enumerate} 
\item   Both $A$ and $B$ have positive Lebesgue measure;
\item Both sets are wandering and their orbits are disjoint. 
\end{enumerate}  We form the sets $X_A=\cup_{n\geq 0}f^{-n}(\cup_{n\geq 1}f^n(A))$ and  $X_B=\cup_{n\geq 0}f^{-n}(\cup_{n\geq 1}f^n(B))$.  They are both invariant and both have positive Lebesgue measure;  this proves that the map is non ergodic. 

For the construction of $A$ and $B$, we use the fact  that  the map $f$ has $N$ asymptotic tracts and $N$ asymptotic values, each of which is a prepole. Therefore, each asymptotic tract is expansively mapped over all $N$ asymptotic tracts under some finite number of iterations. We fix an asymptotic tract and choose some quadrilaterals that, under an appropriate change of coordinate are squares.  We then show that the Lebesgue measure of the set of points whose orbit always lands on the fixed asymptotic tract is positive and that the set wanders to $\infty$;  these are  "almost" Markov chains because the map is expanding. If we choose the quadrilaterals so that the distance between them is large enough,  their orbits will be disjoint since the expansion rate on the quadrilaterals is different.    
\subsection{The behavior of $f$ in a neighborhood of infinity}

The proof of the main theorem depends on a careful study of the behavior of the Nevanlinna function $f$ in a neighborhood of infinity.  (See \cite{CJK4, DK, H, L}.)

 Let $f$ be a Nevanlinna function with Schwarzian derivative  $S(f) =2P(z)$, where the degree of  the polynomial $P(z)$ is $N-2$ and the  leading coefficient of $P(z)$ is the constant $a\in \mathbb{C}^*$.   The Schwarzian equation is related to the linear equation 
$ w''+P(z)w = 0$:  that is,  if $w_1, w_2$ are linearly independent solutions of this equation, then $$S(\frac{a w_1 +b w_2}{cw_1+dw_2})=2P(z).$$   In the special case, $N=2$,  the solutions $w_1, w_2$ can chosen as exponentials whose asymptotic tracts are half planes.   

This linear equation is often studied by making  changes of variable in the asymptotic tracts in order to estimate the behavior of $f$ in them.  The new variable turns the asymptotic tract into a half plane and the Liouville transformation turns the Schwarzian equation into an approximation of the equation in the case where $N=2$. The interested reader can find more details in the discussions in \cite{CJK4} and \cite{L}.   We follow the notation in those discussions.

Denote solutions to the congruence $$\arg a+N\theta \equiv 0\ \ \mod 2\pi$$ by $\theta_i$,
where $1\leq i\leq N$.
The {\em  critical rays} of $f$ are the half lines  $L_{i}=te^{ i\theta_i}, t>0$.  For a small $\epsilon_0>0$, define the sectors
 $$
 S_i=\{z:\ |\arg z-\theta_i|\leq \frac{2\pi}{N}-\epsilon_0\}
 $$for $i=1. \ldots,N$.
Thus the sector $S_i$ contains the critical ray $L_{i}$ and is contained between the critical rays $L_{i-1}$ and $L_{i+1}$.  (By convention, all indices are taken modulo  $N$).   If we denote the boundaries of the sectors by $B_i^{\pm}$, depending on the sign of $\arg z -\theta_i$, the critical ray $L_{i}$ is contained in a critical wedge $wed_{i}$ of angle $2\epsilon_0$ whose boundaries are the rays $B_{i-1}^+$ and $B_{i+1}^-$.

In each sector $S_i$, the function $f$ has a ``truncated solution'' which grows like $exp(z^{N/2})$.    More precisely, let $R>>0$ and $A_R=\{z:\ |z|>R\}$. Then in $S_i$,
 one can define the auxiliary variable
   $$\calz_i(z)=\int^z_{Re^{i\theta_i}} \sqrt{P(s)}ds=\frac{2\sqrt{a}}{N}z^{N/2}(1+o(1)), z\in S_i\cap A_R.$$
       We choose the branch of  $\sqrt{P(s)}$ so that  in the sector $S_i$, for any $z$ on the critical ray $L_{i}$, $\calz_i(z)$ is on the positive real line.   The image of $S_i$ under $\calz_i$ contains a
 sector in the $\calz_i$-plane,
     $$\cals_i=\{\calz_i=\calz_i(z) : \ |\arg \calz_i|<\pi-\epsilon\}\subset \calz_i(S_i)$$ for some  small $\epsilon >0$  depending on $N$ and $\epsilon_0$. In $\mathcal{S}_i$, we define $$\calu_i=\mathcal{S}_i\cap \{\calz_i: \ \Im \calz_i> c\} \text{ and }\call_i=\mathcal{S}_i\cap \{\calz_i: \ \Im \calz_i< -c\}$$ for some $c>>0$.
\begin{figure}
  \centering
  \includegraphics[width=4in]{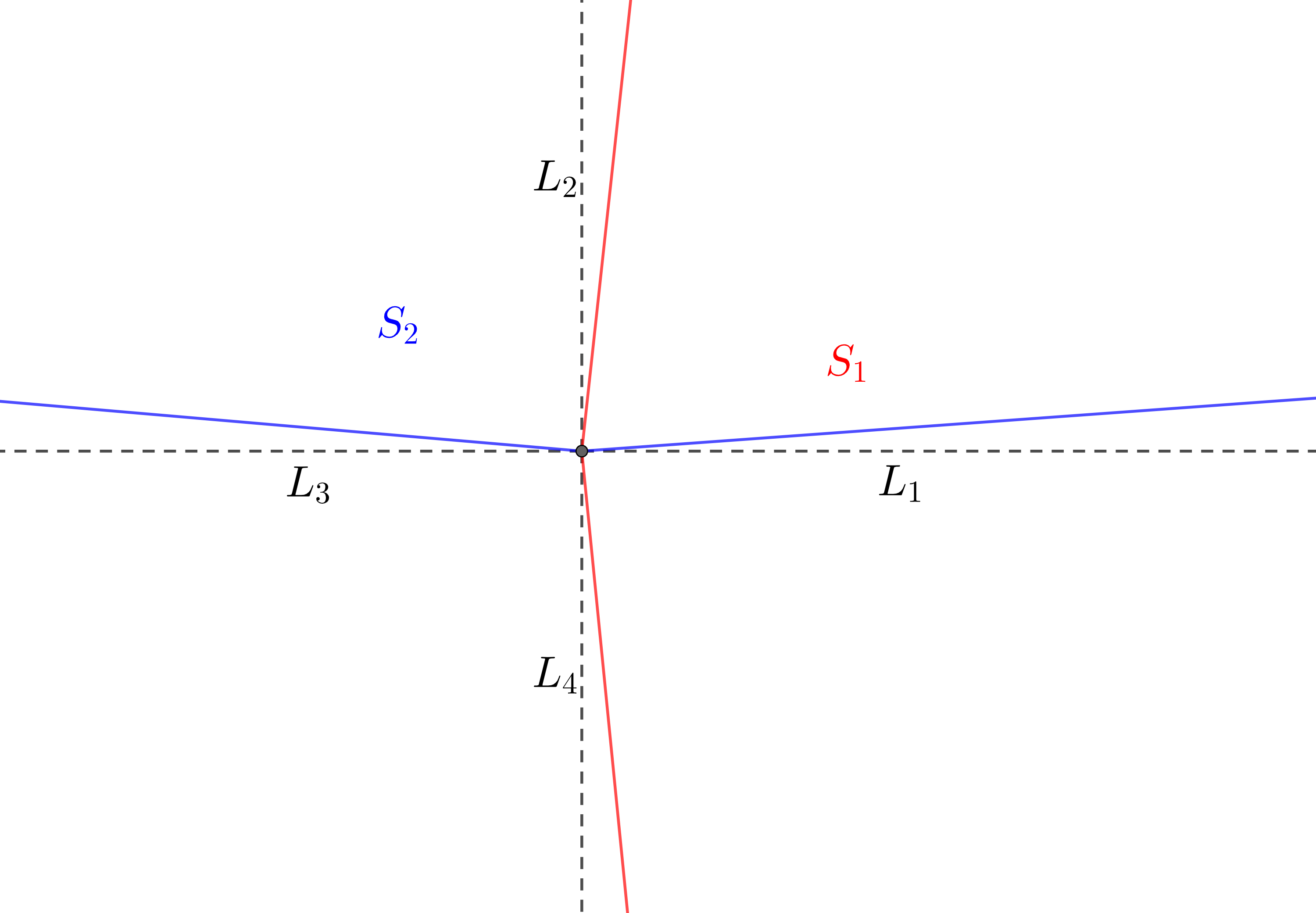}\\
\caption{The critical lines and sectors for $N=4$
}\label{fig:sectors}
\end{figure}

For $z \in S_i$ and $\calz_i=\calz_i(z)$,  we use the Liouville transformation to transform the linear equation to one in the variable $\calz_i$  that has a solution on $\cals_i$ defined by $F(\calz_i(z))=f(z)$.  The map $F(\calz_i)$ can be approximately expressed as
  \begin{equation}\label{map}
  f(z)=F(\calz_i)=\frac{A_ie^{i\calz_i}+B_ie^{-i\calz_i}}{C_ie^{i\calz_i}+D_ie^{-i\calz_i}}.\end{equation}

  The sets  $\calu_i$ and $\call_i$ are respectively asymptotic tracts for the asymptotic values
  $ B_i/D_i  \text{ and }  A_i/C_i$ of $F(\calz_i)$. Therefore, since $$T_i=\calz^{-1}_i(\calu_i) \text{ and } T_{i-1}=\calz^{-1}_i(\call_i)$$ are mapped by $f$ to punctured neighborhoods of the asymptotic values $\la_i$ and $\la_{i-1}$ of $f$,  $\la_i$ and $\la_{i-1}$ are also the asymptotic values of $F$;
  that is,\footnote{Since this holds for each pair of asymptotic values, $\la_i,\la_{i-1}$, it is clear both cannot be infinite.}
  $$
  \frac{B_i}{D_i}=\la_i\quad  \text{ and } \quad \frac{A_i}{C_i} = \la_{i-1}.
  $$

By equation (\ref{map}), on the asymptotic tracts  $\calu_i$ and $\call_i$, the map $F(\calz_i)$ can be expressed as a composition
 $$F(\calz_i)=\begin{cases} M_{i,U}\circ E_{U} \text{ for }\calz_i\in U_i, \\
 M_{i,L}\circ E_{L} \text{ for }\calz_i\in L_i\end{cases}$$ where
$$E_{U}(\calz_i)=e^{2i\calz_i}, \,  \,  E_{L}(\calz_i)=e^{-2i\calz_i}; $$
 $$M_{i,U}(\xi)=\frac{A_i\xi+B_i}{C_i\xi+D_i}, \, \text{ and }\, M_{i,L}(\xi)=\frac{A_i+B_i\xi}{C_i+D_i\xi}.$$
Note that $E_{U}$ and $E_{L}$ are infinite to one universal covering maps  of  $\calu_i$ and $\call_i$ onto the punctured disk
 $D^*=D^*(0, e^{-2c})$.
 Both the M\"obius maps $M_{i,U}$ and $M_{i,L}$ map
$D=D^* \cup \{0\}$  injectively onto neighborhoods $N_i$ and
$N_{i-1}$ of the asymptotic values $\la_i$ and $\la_{i-1}$.
Thus we obtain  factorizations of the truncated solutions in $ T_i$ and $T_{i-1}$
 $$f=M_{i,U}\circ E_U \circ \calz_i  \mbox{  and } f=M_{i,L} \circ E_L \circ \calz_{i}.$$

By hypothesis, $f^{k_i-1}$ and $f^{k_{i-1}-1}$ map $N_i$ and $N_{i-1}$ to neighborhoods $P_i$ and $P_{i-1}$  of the poles $p_i$ and $p_{i-1}$ and so $f$ maps both $P_i$ and $P_{i-1}$ to a neighborhood of infinity.  Because it is  simpler to estimate maps between finite regions,  we introduce a transformation that maps a neighborhood of infinity into a neighborhood  of the origin.  To this end, let $I(z)=1/z$  
  and set  $$f=I\circ g, \text{where } I=\frac{1}{z}, \, g(z)=\frac{1}{f(z)}$$  so that the map $g(z)$ maps both $P_i$ and $P_{i-1}$ into a neighborhood of the origin.
The maps 
$$h_{i, U}(\xi)=g\circ f^{k_{i}-1}\circ M_{i,U}(\xi)     \mbox{ and } h_{i,L}(\xi)=g \circ f^{k_{i-1}-1}\circ M_{i,L}(\xi)$$
both fix the origin and so map the disk $D(0, e^{-c}) $ into a neighborhood of $0$.

\begin{figure}
\includegraphics[width=6in]{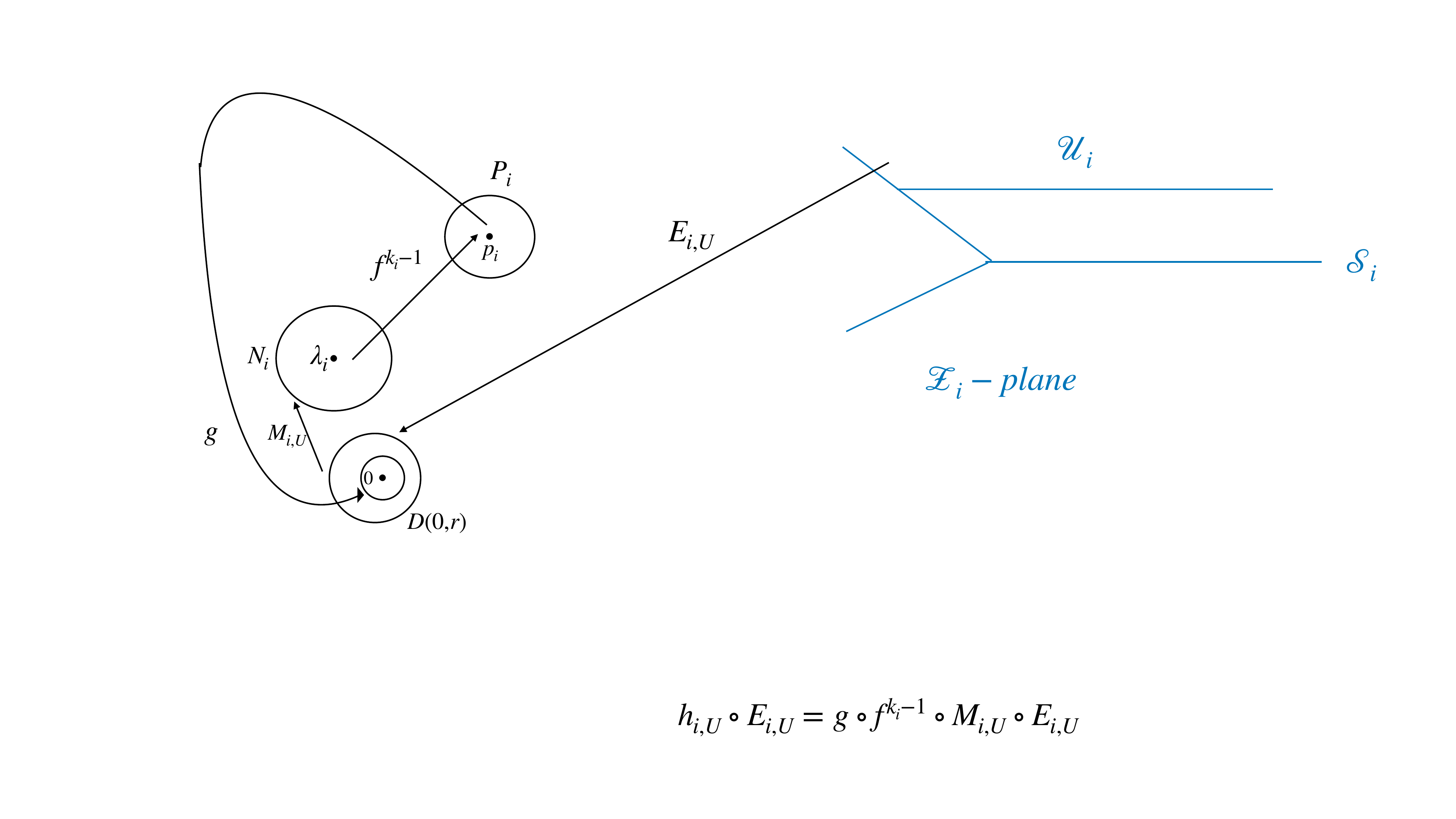}
\caption{The decomposition of $h_{i,U} \circ E_{i,U}$ as a map from the auxiliary   plane to the dynamic plane}
\label{decomph}
\end{figure}

Next,  define the following maps on $T_i$ and $T_{i-1}$ in the sector $S_i$ respectively as:
$$\varphi_{i, U}(z)= f^{k_{i}+1}(z)=I\circ h_{i,U}\circ E_U(\calz_i),\ \  \varphi_{i,L}(z)=f^{k_{i-1}+1}(z)=I\circ h_{i,L}\circ E_L(\calz_i).$$
The maps $\varphi_{i,U}$ and $\varphi_{i,L}$  are compositions of a Mobius map, a univalent map and the exponential map, and they map the respective asymptotic tracts $T_i$  and $T_{i-1}$ onto a neighborhood  $\Omega$ of $\infty$. See figures~\ref{decomph} and \ref{asymtract}.  Next,  for $z\in T_i\cup T_{i-1}\subset S_i$,  we define the maps:

$$
\Phi_i(z)=\begin{cases} \varphi_{i,U}(z), z\in T_{i} \\
\varphi_{i,L}(z), z\in T_{i-1}.
\end{cases}
$$

\begin{figure}
\begin{center}
  \includegraphics[width=6in]{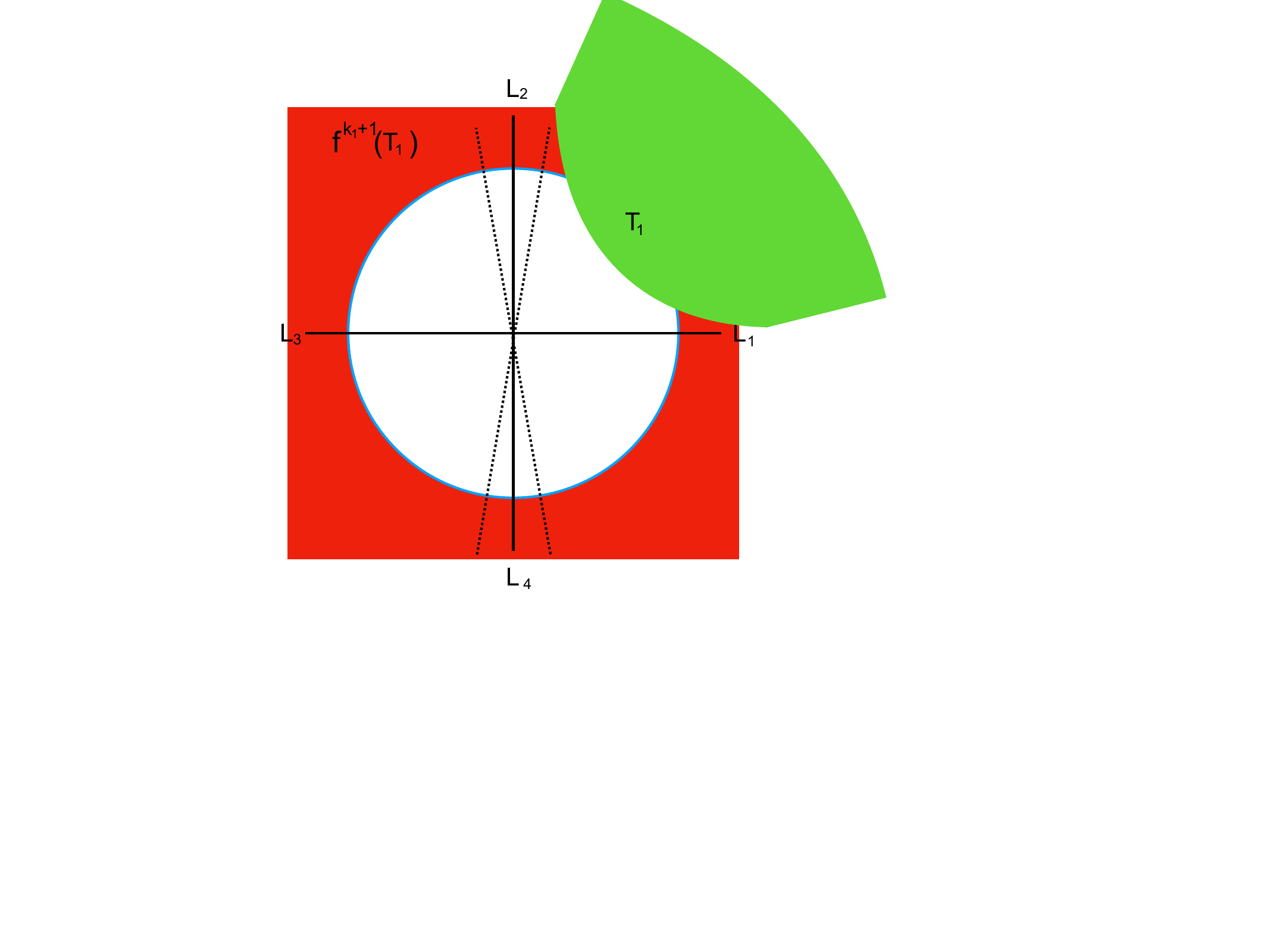}\\
 \caption{The map of the asymptotic tract $T_i$ (green) and its image under $f^{k_i+1}$ (red) }
  \label{asymtract}
  \end{center}
\end{figure}

Note that $T_i\subset (S_i\cap S_{i+1})\cup wed_i$. If $z\in  S_i\cap S_{i+1}$, there are two choices of the auxiliary variables $\calz_i,\ \calz_{i+1}$ for $z$. By the choice of  the branch of $\sqrt{P(z)}$,   $\calz_i(z)=-\calz_{i+1}(z)$ is in the upper half plane.  Therefore, if  
$$\xi= E_{i,U}(\calz_{i}(z))=E_{i+1,L}(\calz_{i+1}(z))$$  then $h_{i, U}(\xi)=h_{i+1, L}(\xi)$; that is, $\varphi_{i,U}(z)=\phi_{i+1,L}(z)$  so that  the map $\Phi_i$ is well-defined.

 Recall that the fixed constant $c>>0$ is the large constant that defines the sets $\call_i$ and $\calu_i$.   Assume it is chosen large enough so that for 
all $1\leq i\leq N$,
\begin{enumerate}
\item $f^j(\lambda_i)\notin \cup_{i=1}^N T_i, j=0, 1 \ldots, k_i-1$.
\item $h_{i,U}(z), h_{i,L}(z)$ are univalent on the filled disk $D(0,r) = D^*(0,r) \cup \{0\},$ where $r=e^{-c}$.
\end{enumerate}

 Define the constants
  $$
  m_{i}=\frac{1}{|h'_{i,U}(0)|}, m'_i=\frac{1}{|h'_{i,L}(0)|}\quad 1\leq i\leq N.
 $$
 and
 $m=\min_{i}\{m_i, m_i'\}, M=\max_{i}\{m_i, m'_i\}$.

\subsection{Basic calculations}
\label{calculations}

 Choose an $\alpha_0>c$, and for integers $k>0$\footnote{ We use $k$ here as an index and $k_i$ as the order of the prepole $\la_i$. This should not cause confusion.},  where $N$ is the number of asymptotic values of $f$, define the sequences of real numbers,  $$\alpha_k=e^{N\alpha_{k-1}}.$$

\begin{lemma}
If $z\in T_i$ and   $\calz_i=\calz_i(z)=x+iy$ with $y>\alpha_0$, then
\begin{equation}\label{lower}
|\Phi_i(z) |\geq m_{i} (e^{2y}-2e^{2c}+e^{4c-2y})
\end{equation}
and
\begin{equation}\label{upper}
|\Phi_{i}(z)|\leq m_{i} (e^{2y}+2e^{2c}+e^{4c-2y}).
\end{equation}

\end{lemma}

\begin{proof}   Suppose that $z\in T_i$ satisfies  $\calz_i=\calz_i(z)=x+iy$  with $ y>c$,  then
 $$\xi = E_U(\calz_i)=e^{2i\calz_{i}} \in D^*(0,r),$$
  $|\xi | =e^{-2y} =  \rho r $ and  $\rho=e^{2c-2y} \in(0,1)$. Since the map $h_{i,U}$ is univalent on $D(0,r)$ and $h_{i,U}(0)=0$,  Koebe's distortion theorem implies
$$\frac{|h'_{i,U}(0)|\rho r}{(1+\rho)^2}\leq |h_{i,U}(\xi)|\leq \frac{|h'_{i,U}(0))|\rho r}{(1-\rho)^2}.$$

Since  $I(z)=1/z$, 
$$\frac{(1-\rho)^2}{\rho r}\frac{1}{|h'_{i,U}(0)|}\leq |I(h_{i,U}(\xi))|\leq \frac{(1+\rho)^2}{\rho r}\frac{1}{|h'_{i,U}(0)|}.$$
Note that $$\frac{(1-\rho)^2}{\rho r}=\frac{1}{\rho r}-\frac{2}{r}+\frac{\rho}{r}=e^{2y}-2e^{2c}+e^{4c-2y},$$ and $$\frac{(1+\rho)^2}{\rho r}=e^{2y}+2e^{2c}+e^{4c-2y}.$$
Inequalities~(\ref{lower}) and ~(\ref{upper}) follow, and the proof of the lemma is complete.
\end{proof}

\begin{lemma}\label{anglelemma}
Suppose $z_1\in T_i$ and  $z_2\in T_j$ where $1\leq i, j\leq N$ (not necessarily distinct), and that $\Phi_i(z_1)\in S_{i'}$ and $\Phi_j(z_2)$ are in $S_{j'} $.
Set $\calz_i(z_1)=x_1+iy_1$ and $\calz_j(z_2)=x_2+iy_2$.  Assume that $N_0$ is sufficiently large and that for $k> N_0$, $y_1\geq \alpha_k$,  $y_2\geq y_1+2\alpha_{k-1}$.    If
 \begin{equation}\label{arg} |(\arg(\Phi_j(z_2))-\theta_{j'})|\geq \frac{4}{N\alpha_k}, \text{ or equivalently } |\sin(\calz_{j'}(\Phi_j(z_2))| \geq \frac{2}{\alpha_k},\end{equation}
then $\Im (\calz_{j'}(\Phi_j(z_2))) \geq  \Im (\calz_{i'}(\Phi_i(z_1)))+2\alpha_{k+1}$.
\end{lemma}

\begin{proof}
By hypothesis $z_1\in T_i$ and $z_2\in T_{j}$ with $y_2>y_1>\alpha_k>\alpha_0$.  Inequality (\ref{upper}) implies  that
\begin{equation}\label{upper2}
\Im (\calz_{i'}(\Phi_i(z_1))\leq |\calz_{i'}(\Phi_i(z_1))|\leq \frac{2|a|^{\frac{N}{2}}}{N}(m_{i} (e^{2y_1}+2e^{2c}+e^{4c-2y_1}))^{\frac{N}{2}}.
\end{equation}
Moreover, by  inequalities (\ref{lower}) and~(\ref{arg}),  when $k$ is large enough,
\begin{equation}\label{long}
\begin{split}
\Im (\calz_{j'}(\Phi_j(z_2)))&\geq |(\Phi_{j}(z_2) \sin(\arg(\calz_j(\Phi_j(z_2)))|\\
&\geq \frac{2|a|^{\frac{N}{2}}}{N}(m_{j} (e^{2y_2}-2e^{2c}+e^{4c-2y_2}))^{\frac{N}{2}}\frac{2}{\alpha_k}.\\
\end{split}
\end{equation}

Note that $y_2>y_1+2\alpha_{k-1}$, and $\alpha_k=e^{N\alpha_{k-1}}$, so that 
\begin{equation}\label{long}
\begin{split}
(m_{j} (e^{2y_2}-2e^{2c}+e^{4c-2y_2}))^{\frac{N}{2}}\frac{2}{\alpha_k}&>2(m_{j} (\frac{e^{2y_1+4\alpha_{k-1}}-2e^{2c}+e^{4c-2y_2})}{e^{2\alpha_{k-1}}})^{\frac{N}{2}}\\
&=2(m_{j} (e^{2y_1+2\alpha_{k-1}}-\frac{2e^{2c}-e^{4c-2y_2})}{e^{2\alpha_{k-1}}})^{\frac{N}{2}}\\
&>(m_{i} (e^{2y_1}+2e^{2c}+e^{4c-2y_1}))^{\frac{N}{2}}+N\alpha_{k+1}/|a|^{N/2}
\end{split}
\end{equation}
 for $k$ large enough, since $\alpha_k\to\infty$ and $y_1\geq \alpha_k$. Therefore $$\Im( \calz_{j'}((\Phi_j(z_2))))\geq \Im( \calz_{i'}(\Phi_i(z_1)))+2\alpha_{k+1}.$$

\end{proof}

Define a set of horizontal strips in $\calu_i$,
indexed by $k >0$  as
$$Hor_k^i=\{\calz_i=x+iy \in \cals_i:\ \alpha_{k}+2\alpha_{k-1}\leq y \leq \alpha_{k+1}-2\alpha_{k} )\}\cap \calu_i.$$
The pull-back of these strips $$
H_k^i = \calz_i^{-1}(Hor_k^i) \subset T_i.
$$ are strips in $T_i$.

The image $\Phi_i(H_k^{i})$ is an annular region $A_i$ in a neighborhood of infinity of the $z$-plane that,  for all $j=1, \ldots, N$, intersects  the sectors $S_j$,  and in particular, the asymptotic tracts $T_j$, the critical lines $L_j$ and  the wedges $wed_j$. See Figure~\ref{asymtract}.  Thus,  the maps $\calz_j$'s are well-defined for all $1\leq j\leq N$.     Define  the map $\Psi_{i,j}$ on the $\calz_i$-plane by the functional equation
$$\calz_j \circ \Phi_i = \Psi_{i,j} \circ \calz_i.$$

\medskip
 The maps $\Psi_{i,j}$ are infinite to one.  To create  regions of injectivity, divide each horizontal strip  $Hor_k^i$ into infinitely many rectangles $Rect_{k,n}^i $ of width $\pi$  and define $N$ disjoint sub-rectangles $Rect_{j,k,n}^i$ in each rectangle $Rect_{k,n}^i$
as:
$$Rect_{j,k,n}^i=\{\calz_i=x+iy:\  x_{j+1}^i+n\pi+3/(N\alpha_k)\leq x\leq x_{j}^i+n\pi-3/(N\alpha_k)\}$$
where $x_j^i$ satisfies $2x_j^i+\arg(h_{i,U}'(0))=-\theta_j$. Recall that $\theta_j$ is the $j$-th solution of $\arg a+N\theta\equiv 0 \mod 2\pi$.

  \begin{remark} Under the map $\calz_j^{-1}\circ \Psi_{i,j}$, each rectangle $Rect_{k,n}^i $ maps injectively onto the annulus $A_i$ in the $z$-plane. The vertical lines $x=x_j^i  + n\pi$ in $Hor_k^i$  are preimages of  the critical rays $L_{j}$.  The rectangles $Rect_{j,k,n}^i$ in each $Rect_{k,n}^i $ injectively map onto the disjoint asymptotic tracts in the annulus.
\end{remark}

\begin{figure}
  \centering
  \includegraphics[width=4in]{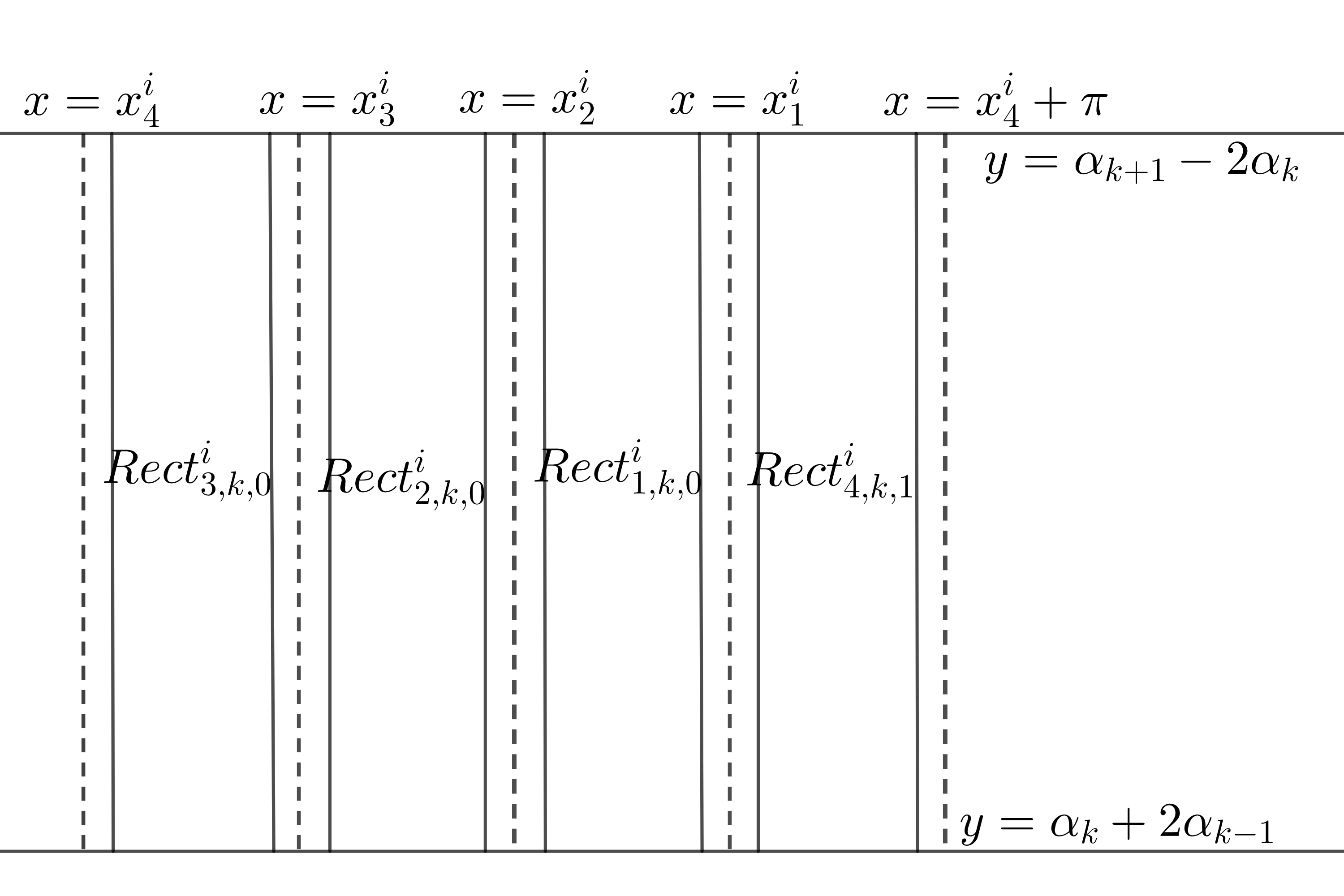}\\
  \caption{$Hor^i_k$ for $N=4$}  
\end{figure}

The next lemma states that in the $\calz_i$-plane, for any point $x+iy$ in these sub-rectangles $Rect_{j,k,n}^i$, the inequality (\ref{arg}) in Lemma \ref{anglelemma} is satisfied.  In particular, the sub-rectangles are mapped away from the critical rays and into an asymptotic tract. 

\begin{lemma} \label{largek}  There exists an integer $N_0$ such that if $k>N_0$, and if $\calz_i (z) \in Rect_{j,k,n}^i$ for some  $i$, then  for each $n$ and for all $j =1, \ldots, N$, 
$$|\arg(\Psi_{i,j}(z))|\geq \frac{2}{\alpha_{k}}.$$
\end{lemma}

\begin{proof}

For the proof, fix  $i$ and $n$  and, for readability,  omit writing them.  Set
 $$M_j=  \{\calz=x_{j}+iy, \, y\geq \alpha_0  \}, $$
 $$L_{j,k}= \Big\{\calz=\Big(x_{j+1}+\frac{3}{N\alpha_k}\Big)+iy, \, y >\alpha_0 \Big\} \text{ and } $$
 $$R_{j,k}=  \Big\{\calz=\Big(x_{j}  - \frac{3}{N\alpha_k}\Big)+iy, \, y>\alpha_0 \Big\}.  $$
  Note that   the vertical lines $L_{j,k}$ and $R_{j,k}$ contain the respective left and right vertical borders of $Rect_{j,k}$ and the line $M_j$ lies in the complementary region between $R_{j,k}$ and $L_{j-1,k}$.

  We claim that
\begin{equation}\label{imaginarg}
\begin{split}
   \text{  if }  \calz  \in L_{j,k}  \text{ then } \theta_{j+1} - \arg(\Phi_i(\calz_i^{-1}(\calz)))>\frac{4}{N\alpha_k} \text{ and } \arg (\Psi_{i,j+1}(\calz)) \leq  -\frac{2}{\alpha_k};\\
   \text{ if  } \calz  \in R_{j,k}, \text{ then } \arg(\Phi_i(\calz_i^{-1}(\calz)))-\theta_{j}>\frac{4}{N\alpha_k} \text{ and   } \arg (\Psi_{i,j}(\calz))\geq \frac{2}{\alpha_k}. \,
\end{split}\end{equation}

 This says that  both $\Phi(\calz^{-1}(L_{j,k}))$ and $\Phi(\calz^{-1}(R_{j,k}))$
are bounded away from the images of the critical rays $L_i$ and $L_{i+1}$;  equivalently, the images of $L_{j,k}$ and $R_{j,k}$ under $\Psi$ are bounded away from the real line, as are the points in $Rect_{j,k}$, and proves the lemma.

\begin{figure}
 \centering
 \includegraphics[width=4in]{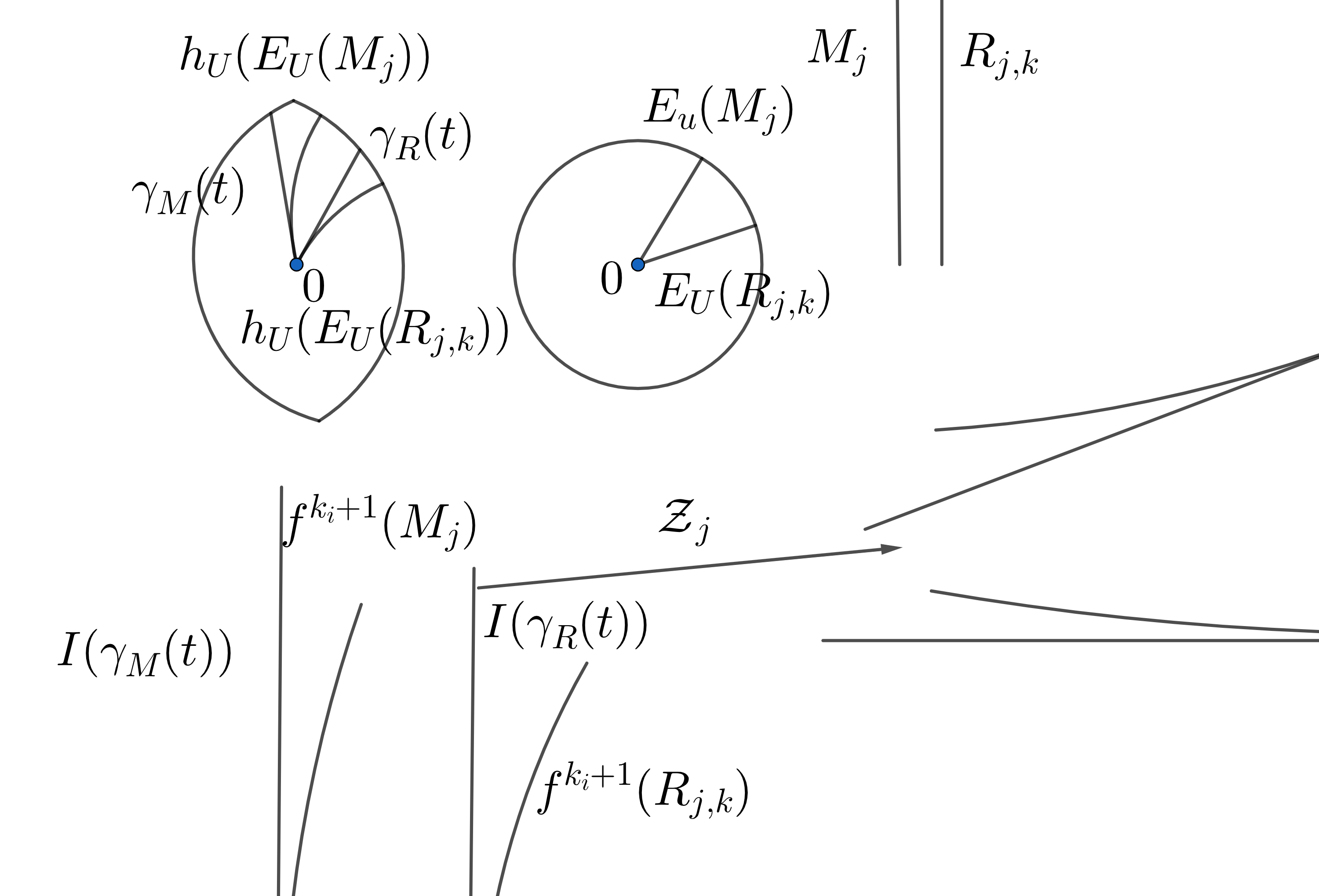}
  \caption{Rays and curves in lemma~\ref{largek}}
 
\end{figure}
To prove the claim, first consider the images of the half-lines $R_{j,k}$ and $M_{j}$ under $E_U$:

$$E_U(R_{j,k})=e^{-2y}e^{2i(x_{j} - \frac{3}{N\alpha_k}}) \mbox{ and } E_U(M_{j})=e^{-2y}e^{2i(x_{j})},\ y\geq \alpha_0.$$
They are line segments meeting at $0$ and  the angle between them satisfies $$\arg(E_U(R_{j,k}))-\arg(E_U(M_{j}))=-\frac{6}{N\alpha_k}.$$
Next, since $h_{U}: D(0,\rho) \to \mathbb{C}$ is conformal, the images $h_{U}(E_U(M_j))$ and $h_{U}(E_U(R_{j,k}))$ are curves with the initial point $0$. Let $\gamma_M(t)$ and $\gamma_R(t)$ be the lines tangent to these curves at $0$, respectively. Then
\begin{equation}\label{tangent}
\begin{split}
 \gamma_M(t)=te^{i(\arg(h_{U}'(0))+2x_{j})}, \; t\geq 0,  \text{  and }\\
\gamma_R(t)=te^{i(\arg(h_{U}'(0))+2x_{j}-\frac{6}{N\alpha_k})}, \; t\geq 0.
 \end{split}
\end{equation}
That is, the angle between $\gamma_M(t)$ and $\gamma_R(t)$ is equal to $6/(N\alpha_k)$.

Suppose   $\calz$ is on the line $R_{j,k}$, $\xi=E_{i,U}(\calz)$,  and $|\xi|=\rho r $,  with $|\rho|<e^{2c-2\alpha_k}<<1$.  Then,  if $k$ is large enough, by Koebe's distortion theorem,
  $$|\arg (h_{U}(\xi))-(\arg h_{U}'(0)+2x_{j}-\frac{6}{N\alpha_k}))|<2\ln(\frac{1+\rho}{1-\rho})<\frac{1}{N\alpha_k}.$$
  That is, $h_{U}(\xi)=r_0e^{i \theta}$ for some $r_0>0$, where
\begin{equation}\label{argdifference}
\theta-(\arg h_{U}'(0)+2x_{j})\geq \frac{4}{N\alpha_k}.
\end{equation}

If  $w = r_0 e^{i(\arg h_U'(0)+2x_{j})}$, then
$$I(w)=\frac{1}{r_0}e^{-i(\arg h_{U}'(0)+2x_{j})} \text{ and } \Phi_i(z)=I(h_{U}(\xi))=\frac{e^{-i\theta}}{r_0}.$$

    By the definition of $x_{j}$,  $\arg(I(w))=\theta_{j}$, and by inequality (\ref{argdifference}),\\
     $\arg(\Phi_i(\calz_i^{-1}(\mathcal{Z})))-\theta_{j}>4/(N\alpha_k)$.  Note that since $\calz$ is an arbitrary point in $R_{j,k}$,  by the definitions of $\theta_j$ and  $I(\gamma_M(t))$, it is mapped to the positive real line under the map $\mathcal{Z}_j$.  Since $\Psi_{i,j}=\mathcal{Z}_j\circ\Phi_i$,  we conclude
      $$\arg(\Psi_{i,j}(\calz))\geq 2/\alpha_k.$$

One can check, using similar arguments, that at each point  $z$ such that   $\calz(z)$ is on the line $L_{j,k}$,
 $$\theta_{j+1}-\arg(\Phi_i(\calz_i^{-1}(\calz)))\geq 4/(N\alpha_k).$$

 This completes the proof of the claim and hence the lemma. 
 \end{proof}

\begin{lemma}\label{mapping props}  Suppose $k>N_0$.  If   $\calz_i \in Rect^i_{j,k,n}$ for some $n$, then $ \Psi_{i,j}(\calz_i) \in Hor_{k+1}^j$.
\end{lemma}

\begin{proof} Suppose $k>0$ and $\calz_i\in Rect^i_{j,k,n}$.  Then if
$\calz_i=x+iy$,
$$
\alpha_k+2\alpha_{k-1} \leq  y \leq \alpha_{k+1}-2\alpha_k.
$$
By
  inequality (\ref{upper}),
   $$\Im  \Psi_{i,j}(\calz_i) < | \Psi_{i,j}(\calz_i)|\leq\frac{2\sqrt{|a|}}{N} \Big(m_{i}(e^{2\alpha_{k+1}-2\alpha_k}+2e^{2c}+e^{4c+2y})\Big)^{N/2}<\alpha_{k+2}-2\alpha_{k+1}.$$

Since $\calz_i \in Rect^i_{j,k,n}$,  by Lemma~\ref{largek},
$|\arg(\Phi_{i,j}(\calz))|>2/\alpha_k$.  Because $\calz_i \in Rect^i_{j,k,n}$,  $\Im \calz_j  \geq \alpha_k+2\alpha_{k-1}$; thus Lemma \ref{anglelemma}  implies 
$$\Im (\Psi_{i,j}(z))) \geq 2\alpha_{k+1}>\alpha_{k+1}+2\alpha_k;$$
 Therefore  $\Psi_{i,j}(z))\in Hor_{k+1}^j$.
\end{proof}

\section{ Disjoint wandering sets of positive Lebesgue measure}\label{wanderingsets}
    To construct the wandering sets it is more convenient to  work with the maps $\Psi_{i,j}$ on the auxiliary planes and then pull back.  The first step is to estimate the expansion factor.

\begin{prop}\label{expanding}  Fix the point  $\calz_i^*$ in the $\calz_i$ plane  and assume that   $\Psi_{i,j}(\calz_i^*) \in Hor_k^j$;  then
$$
|(\Psi_{i,j})'(\calz_i^*)| \geq \frac{\alpha_{k}}{4\pi}.
$$
\end{prop}

\begin{proof}
 Set $\calz_j^*=\Psi_{i,j}(\calz_i^*) \in Hor_{k}^j$; then
$\alpha_{k}+2\alpha_{k-1}\leq  \Im (\calz_j^*) \leq \alpha_{k+1}-2\alpha_{k}$.

Let  $D_k=D(\calz_j^*, \alpha_{k}/2)$. By our choices of $c$ and $\alpha_0$,   the orbits of all  the $\la_i$'s are outside its preimage, $\calz_j^{-1}(D_k)$.  It follows that
    there is a univalent branch, $G$  of $\Psi_{i,j}^{-1}$ defined on $D_k$. By Koebe's $\frac{1}{4}$-theorem,
$$D\Big(\calz_i^*, \frac{|G'(\calz_j^*)|}{4}\frac{\alpha_k}{2}\Big)\subset G( D_k).$$
Because $\Psi_{i,j}$ is univalent  on $G(D_k)$, the radius of $G(D_k)$  is less than $\pi/2$; that is,
$$
\frac{|G'(\calz_j^*)|}{4}\frac{\alpha_k}{2}\leq \frac{\pi}{2}.
$$
This implies that
$$
|\Psi_{i,j}'(\calz_i^*)|\geq \frac{\alpha_k}{4\pi}.
$$
\end{proof}

Define a family  $\mathcal{G}^i=\{S^i_{m,n}; m,n \in \mathbb{Z} \}$, where $S^i_{m,n}$ is a square bounded by the straight lines in the $\calz_i$-plane given by
$$x=x^i_N+n\pi, \;x=x^i_N+(n+1)\pi, \; y=m\pi, \;y=(m+1)\pi .$$

For each $k>0$, let $\mathcal{G}^i_k\subset \mathcal{G}^i$ be the collection of $S_{m,n}^i\subset Hor_k^i$.  It will be convenient to use the index $\delta^i$ where $S_{\delta^i}=S_{m,n}^i$.  Note that since $x_N^i$ is fixed, each such square lies between the vertical lines
$$
V_n=V(x^i_N+n\pi)=\{\calz=x^i_N+n\pi+iy, y > \alpha_0\} $$ and $$ V_{n+1}=V(x^i_N+(n+1)\pi)=\{\calz=x^i_N+(n+1)\pi+iy), y > \alpha_0\}.
$$

Denote $\mathcal{G}=\cup_{i=1}^N\cup_{k>0} \calg_k^i$.

For a square $S_{\delta^i}\in \mathcal{G}^i_k$, define the rectangles $\calr_{\delta^i}^j=S_{\delta^i}\cap Rect_{j,k,n}^i$. It follows from the definition of $Hor_{k}^i$    that
 $$\frac{m(S_{\delta^i}\setminus (\cup_{j=1}^N\calr_{\delta^i}^j))}{m(S_{\delta^i})}\leq \frac{6}{\alpha_k}.$$

By Lemmas~\ref{largek} and~\ref{mapping props}, for each $j=1,\ldots, N$,  $ \Psi_{i,j} $ maps the rectangle $\calr_{\delta^i}^j$ onto a simply connected domain contained in the horizontal strip  $Hor_{k+1}^j$.  

Let $\calu^{j}(S_{\delta^i})$ denote the union of all the squares $S_{\delta^j}$
entirely contained in the interior of $\Psi_{i,j}(\calr_{\delta^i}^j)$:
   $$\calu^j(S_{\delta^i}) = \cup S_{\delta^j} \text{  where  } S_{\delta^j}\subset {\rm Interior} ( \Psi_{i,j}(\calr_{\delta^i}^j))\subset Hor_{k+1}^j. $$ 
  Note that because all the squares have side length $\pi$,
    $$\rm{dist}(\partial \Psi_{i,j}(\calr_{\delta^i}^j), \partial \calu^j(S_{\delta^i}))\leq \sqrt{2}\pi.$$  Otherwise,   more squares could be added inside $\Psi_{i,j}(R^i_{\delta_j})$.

\begin{figure}
 \centering
 \includegraphics[width=4in]{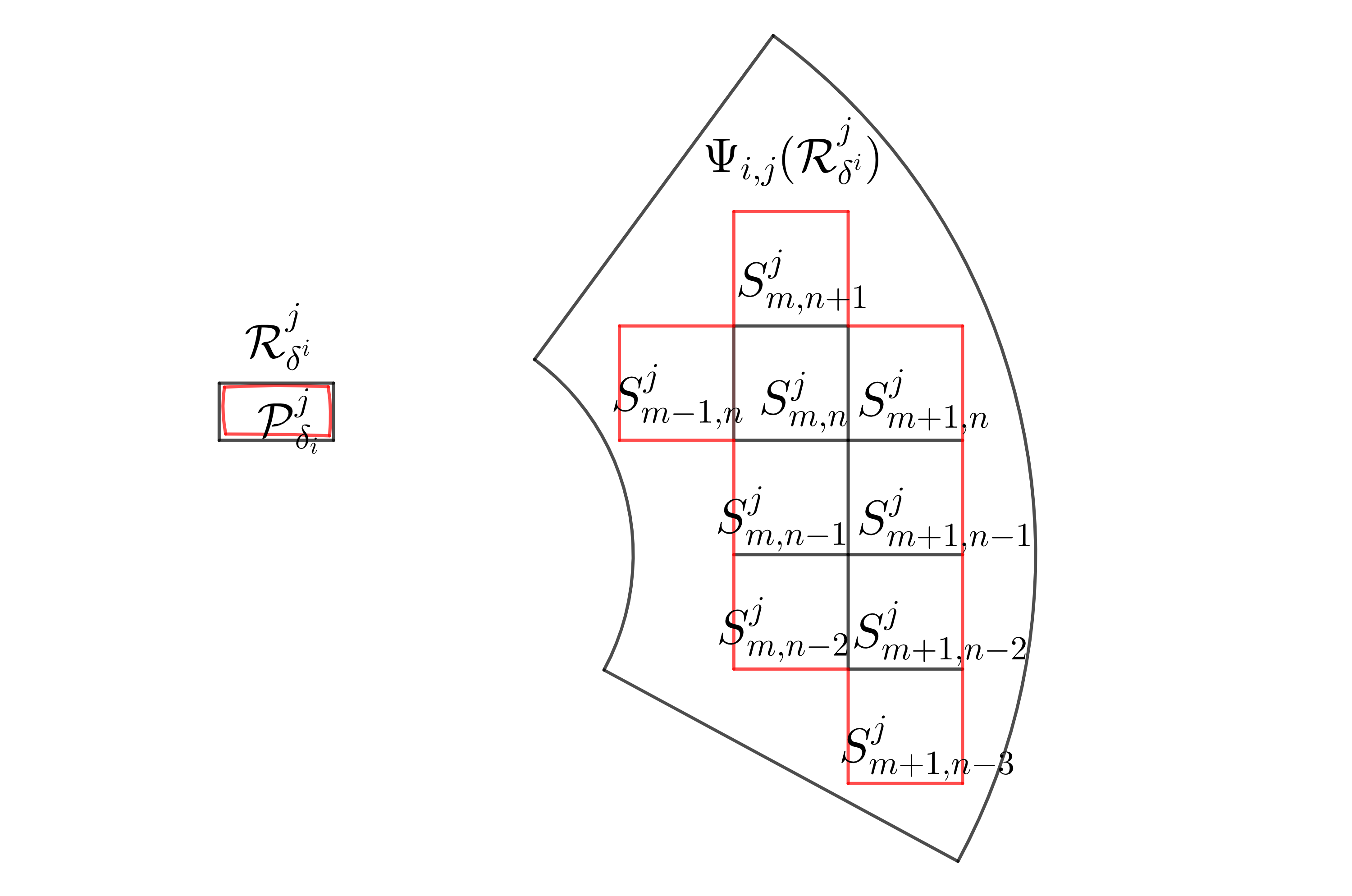}
  \caption{The definition of $\mathcal{P}_{\delta^i}^j$}
   \label{markov}
\end{figure}

For each $j=1, \ldots,N$, let $\calp_{\delta^i}^j=\calp^j_{\delta^i}(S_{\delta^i})=\Psi_{i,j}^{-1}(\calu^j(S_{\delta_0^j}))$. Since it is the pullback of a union of squares, the following inequality shows that each $\calp^j_{\delta^i}$ is a domain  whose area is very close to that of the rectangle   $\calr_{\delta^i}^j$, see Figure~\ref{markov}. 
 Proposition~\ref{expanding} shows that the expansion factor for $\Psi_{i,j}$ on $S_{\delta^i}$ is at least $\alpha_k/(4\pi)$; this says that
 $${\rm dist}(\partial \calp^j_{\delta^i}, \partial \calr_{\delta^i}^j) < \frac{\rm{dist}(\partial \Psi_{i,j}(\calr_{\delta^i}^j), \partial \calu^j(S_{\delta^i}))}{\frac{\alpha_k}{4\pi}}\leq  \frac{ 4\pi^2 \sqrt{2}}{\alpha_k}$$
   so that the set  $\calr_{\delta^i}^j\setminus \calp^j_{\delta^i}$ is contained in a $4\pi^2 \sqrt{2}/\alpha_k$  neighborhood of
$\partial \calr_{\delta^i}^j$.  Thus for the rectangles $\calr_{\delta^i}^j$,
 $$\frac{m(\calr_{\delta^i}^j\setminus \calp^j_{\delta^i})}{m(\calr_{\delta^i}^j)}\leq \frac{4\sqrt{2}\pi^2}{\alpha_k}$$
and for the full square $S_{\delta^i}$,

\begin{equation}
\label{approxsq}
\frac{m(S_{\delta_0^i}\setminus (\cup_{i=j}^N\calp^j_{\delta^i}))}{m(S_{\delta_0^i})}\leq \frac{4 \sqrt{2}\pi^2+6}{\alpha_k}.
\end{equation}
Thus each $\calp^i_{\delta_j}$ is ``almost'' equal to the rectangle $\calr_{\delta^i}^j$ and $\cup_{i=1}^N\calp^j_{\delta^i}$ is ``almost" equal to the square $S_{\delta^i}$.   The set $\calu^j(S_{\delta_0^i})=\Psi_{i,j}(\cup_{j=1}^N\calp^i_{\delta_j})$ is a union of actual   squares, which by Lemma~\ref{anglelemma}, lie in $Hor^j_{k+1}$.

The process of using the maps $\Psi_{i,j}$ and their inverses to push forward and pull back can be iterated any number of times starting from any square $S_{\delta^i}$ in some $Hor_k^i$.    The Lebesgue measure of the resulting pullbacks to $S_{\delta^i}$ needs to be estimated. In order to do this, we need to introduce some more notation for iterated maps.

Given $l \in \NN$, let $\boldsymbol{\iota}=(\delta^{i_0}, \delta^{i_1}, \ldots, \delta^{i_l})$,  where $i_0, i_1, \ldots, i_l, \in \{1, \ldots, N\}$.   Let  $\boldsymbol{\sigma}=\{j_0,\ldots, j_l\}$ where  $j_0, \ldots, j_l \in \{1, \ldots, N\}$ and $j_0=i_1, \ldots, j_{l-1}=i_l$.
 Denote the composition $\Psi_{i_1,i_2} \circ \Psi_{i_0,i_1} $ by $\Psi_{i_0,i_1,i_2}$ and for each $l$, inductively set
$$\Psi^l=\Psi_{i_0,i_1, \ldots i_l}.$$

 Define $\calp^{\boldsymbol{\sigma}}_{\boldsymbol{\iota}}$ as the set consisting of all points in the square $S_{\delta^{i_0}}\subset \calz_{i_0}$ whose orbit under $\Psi^l$ lies in the sequence of quadrilaterals  $\calp_{\delta_j}^{i_{j+1}}(S_{\delta^{i}_{j+1}})$, for  $j=0, \ldots, l-1$.

Denote the family of indices $({\boldsymbol{\iota}},{\boldsymbol{\sigma}})$ for which  $\calp_{\boldsymbol{\iota}}^{\boldsymbol{\sigma}}$ has a nonempty interior, and is contained in $Hor_k^{i_0}$,  by $\mathcal{I}_k^l$ .   For readability,
identify   the index $(\boldsymbol{\iota},{\boldsymbol{\sigma}})\in \mathcal{I}_k^l $ with the set  of  points in   $\calp_{\boldsymbol{\iota}}^{\boldsymbol{\sigma}}$.

  Let $N_0$ be chosen as in Lemma~\ref{largek} and assume $k>N_0$.   Set $\mathcal{I}^l=\cup_{k\geq N_0} \mathcal{I}_k^l$. By definition, for each
   $\calp^{\boldsymbol{\sigma}}_{\boldsymbol{\iota}}\in \mathcal{I}^l$,  its image under $\Psi^{l+1}$ is a union of  squares in $Hor_{k+l+1}^{j_l}$.

Set $$Y^l=\cup_{(\boldsymbol{\iota},\boldsymbol{\sigma})\in \mathcal{I}^l}\calp_{\boldsymbol{\iota}}^{\boldsymbol{\sigma}}.$$ Note that, by definition, $Y^{l+1}\subset Y^l$.  
Let \begin{equation}
\label{yinf}
Y^\infty= \cap_{l=0}^\infty Y^l.
\end{equation}

\begin{lemma}\label{meas Y}
There exists a constant $C>0$ such that if $K\in\mathcal{I}^l_k$, then $$\frac{m(K \setminus Y^{l+1})}{m(K)}\leq \frac{C}{\alpha_{|k|+l+1}}.$$
\end{lemma}

\begin{proof}
By definition,  if  $K$ is the set defined by $\calp^{\boldsymbol{\sigma}}_{\boldsymbol{\iota}}\in \mathcal{I}^l_k$, then $\Psi^{l+1}(K)  $  is a union of squares $\cup_{\delta^j\in \Delta}S_{\delta^j}$ in $Hor_{k+1+l}^{j_l}$ for some family $\Delta$.   Denote this union by $W=\Psi^{l+1}(K)$. Recall that each square $S_{\delta^j}$ contains the sets    $\mathcal{P}_{\delta^j}^i, i=1, \ldots, N$, each of which is mapped to a square.    Therefore, 
 $$\Psi^{l+1}(W\cap Y^{l+1})=\cup_{j_l\in \Delta} \cup_{i=1}^N \calp_{S_{\delta^{j_l}}}^i.$$ 
 The union of the orbits of the asymptotic values $\la_i$ is a finite set.  Since the points   $\calz_{j_1}^{-1}(Hor_{k+1+l}^{j_l})$ lie in the asymptotic tract of $\la_i$, we can choose $k$ so that they are far from its orbit.   In fact, we can choose it so that there is a neighborhood of $\calz_{j_1}^{-1}(Hor_{k+1+l}^{j_l})$ of radius $2\pi$ disjoint from the orbit of asymptotic values.  It follows that 
  the branches of  $\Psi^{-(l+1)}$ are well defined on this  neighborhood.
 
   Denote the branch that maps $W$ back to $K$ by $\Xi$ and  apply the Koebe distortion theorem.  It  says that  for $\zeta, \xi\in W$,  there exists a constant $C_{0}$ independent of $\Xi $, such that
$$\frac{|\Xi'(\zeta)|}{|\Xi'(\xi)|}\leq C_{0}.$$
Then $$\frac{m(W\setminus Y^{l+1})}{m(W)}=\frac{m(\Xi(\cup (S_{\delta^{j_l}}\setminus \cup_{i=1}^N \calp_{\delta_{j_l}}^i)))}{m(\Xi(\cup S_{\delta^{j_l}}))}\leq C_0^2\frac{\sum m(S_{\delta_{j_l}}\setminus \cup_{i=1}^N \calp_{\delta^{j_l}}^{i})}{\sum m((S_{\delta^{j_1}})}\leq \frac{C}{\alpha_{|k|+l+1}},$$ for some constant $C=(4\sqrt{2} \pi ^2+6) C_{0}^2>0$.  Note that the last inequality follows from  Equation(\ref{approxsq}). 
\end{proof}

\begin{lemma}
\label{meas square}

If  $N_0>0$ is the integer defined in Lemma~\ref{largek}, then for any square  $S_{\delta^i} \in \cup_{k\geq N_0}\mathcal{G}^i_{k}$, $m(S_{\delta^i}  \cap Y^\infty)>0$.
\end{lemma}

\begin{proof}
Suppose that $S=S_{\delta^i}\in Hor^i_{k}$ for some $k>N_0$.  For each $l \geq 0$, set $S^l=S\cap Y^l$.     According to the previous lemma,
$$m(S^l\setminus S^{l+1})\leq \frac{m(S^l)C}{\alpha_{|k|+l}}.$$
 Equivalently,  $$m(S^{l+1})\geq m(S^l)(1-\frac{C}{\alpha_{|k|+l}}).$$
Since $$\sum_{l=0}^\infty\frac{1}{\alpha_{|k|+l}}\leq \sum_{l=0}^\infty \frac{1}{\alpha_{|k|}^l}$$  and the right side is convergent, it follows that
$$
0<C_1=\prod_{l=0}^\infty\Big(1-\frac{C}{\alpha_{|k|+l}}\Big) <\infty.
$$
Therefore, for $S^{\infty} =S^0\cap Y^{\infty}$,
$$
m(S^\infty)\geq m(S)\prod_{l=0}^\infty\Big(1-\frac{C}{\alpha_{|k|+l}}\Big)\geq C_{1} m(S)>0.
$$

\end{proof}

Now we are ready to complete the proof of  Main Theorem 1.

\begin{proof}[Proof of Main Theorem 1.]
Let $k\geq N_0$, where $N_0$ is the large constant defined in Lemma~\ref{mapping props}.  Let $\cali^{l}$ be defined as above.  From the definitions of $Hor_{k}^{i_0}$, and by Lemma~\ref{meas square}, we can find two squares $S_{\delta^{i_0}}\not=S_{\delta^{i_0}}'\in Hor_k^{i_0} $ such that
\begin{enumerate}
\item $m(S_{\delta^{i_0}} \cap Y^\infty)>0$ and $m(S_{\delta^{i_0}}' \cap Y^\infty)>0$;
\item  Every  pair  of points $(\calz,\calz') \in  (S_{\delta^{i_0}}, S_{\delta^{i_0}}'), $  satisfies
$|\Im \calz|  - |\Im \calz'|\geq 2\alpha_{k-1}$.
 \end{enumerate}

Let $W_1= S_{\delta^i}\cap Y^\infty$ and $W_2= S_{\delta^i}'\cap Y^\infty$. By  construction,  for any $l\geq 0$,
\begin{enumerate}
\item
 $\Psi^l (W_1)\subset Hor^{i_1}_{k+l}$, $\Psi^l (W_2) \subset  Hor^{i_2}_{k+l}$ for some $i_1, i_2\in \{1, \ldots, N\}$;
\item for all pairs $(\calz_{i_1}\in \Psi^l(W_1),  \calz_{i_2}\in \Psi^l(W_2))$, $|\Im \calz_{i_1}|  - |\Im \calz_{i_2}|\geq 2\alpha_{k+l-1}$.
\end{enumerate}
The first assertion shows that both $W_1$ and $W_2$ are wandering sets and the second shows that their forward orbits are disjoint. 
Note that $\calz_{i_0}$ is analytic, so both $\calz_{i_0}^{-1}( W_{1}), \calz_{i_0}^{-1}( W_{2})$  have positive Lebesgue measure.

 Let
 $$
 E_1=\cup_{m=0}^{\infty} f^{-m}(\cup _{n= 0}^\infty f^n(\calz_{i_0}^{-1}(W_1))) \supset \calz_{i_0}^{-1}( W_{1})
 $$
 and
 
 $$
 E_2=\cup_{m=0}^{\infty} f^{-m}(\cup _{n= 0}^\infty f^n(\calz_{i_0}^{-1}(W_2)))\supset \calz_{i_0}^{-1}(W_{2}).
 $$
 Both have positive Lebesgue measure in the Julia set of $f$. They are mutually disjoint and completely invariant.
 This implies that $f$ acts non-ergodically on its Julia set  completing  the proof of the Main Theorem.

\end{proof}

\section{Main Theorems 2 and 3}~\label{ols}

 Let $P_f=\cup_{i=1}^N\{\la_i, f(\lambda_i), \ldots, f^{p_i-1}(\la_i), \infty\}$ be the post-singular set.  As we proved above,  $f$ is not ergodic.  From the extended dichotomy of Bock \cite{Boc}, we know that for almost every point $z\in \CC$, $\lim_{n\to \infty} d(f^n(z), P_f)\to 0$.  This implies that the $\omega$-limit set $\omega(z) \subset P_{f}$ and in terms of the auxiliary variables, it says that for each such $z$, 
    there exists a fixed $i \in \{1, \ldots, N\}$,  a sequence $n_k\to \infty $, and an $n_{k_0}$ such that for all $n_k \geq n_{k_0}$,    $f^{n_k}(z)  $  is in the asymptotic tract $T_i$ and $\Im \calz_{i}(f^{n_k}(z)) \to \infty$.

The set   $Y^\infty$ was defined in (\ref{yinf}) as a  subset of $\cup_{i=1}^N \calz_i$ and we showed in Lemma~\ref{meas square} that it has positive Lebesgue measure.    In order to prove  Main Theorem 2, which is a result about points in $\CC$, we need to pull $Y^{\infty}$ back  to $\CC$.  To do this,   define
$E^{\infty}= \cup_{i=1}^N \calz_i^{-1}(Y^{\infty} \cap \calz_i)$ and set $E = \cup_{n \in \NN} f^{-n}(E^{\infty})$.

 Since  $E^\infty\subset E$ and $m(Y)>0$, pulling back, $m(E)>0$. In fact, more is true.

  \begin{lemma} \label{C-E meas 0} The set $E$ in $\CC$  has full Lebesgue measure;  that is, $m(\CC\setminus E)=0$.
\end{lemma}

\begin{proof} Assume the contrary,  $m(\mathbb{C}\setminus E)>0$,  and let $z \in \CC \setminus E$ be a Lebsgue density point.  As above, there is an $ i \in \{1, \ldots, N \}$ and  a sequence $n_k$ such that $\Im \calz_{i}(f^{n_k}(z)) \to \infty$ as $n_k\to \infty$.  Let $k_0$ be the smallest integer such that $f^{k_0}(z)$ lies in the asymptotic tract $T_i$ and $\Im \calz_i(f^{n_{k_0}}(z)) > \alpha_0$. 

 Let $l_k < \infty$ be a sequence such that
$$\alpha_{l_k}-2\alpha_{l_{k}-1}\leq \Im \calz_{i}(f^{n_k}(z)) \leq \alpha_{l_{k}+1}-2\alpha_{l_k}.$$
Note that these inequalities define a set of strips $\widetilde{Hor^{i}_{l_{k}}}$ that  is different from, but intersects the strips $Hor^{i}_{n_k}$.

 Let $\calz^{n_k}=\calz_{i}(f^{n_k}(z))$, and let $F_{n_k}$ be the  branch of $f^{-n_k}\circ \calz_{i}^{-1}$ that maps  $\calz^{n_k}$ to $z$. 
Let $Q^{n_k}$  be the square centered at $\calz^{n_k}$  with side length $8\alpha_{l_{k}-1}$.\footnote{This is not one of the squares in $\calg^{i}_{n_k}$.}
Then $$m(Q^{n_k} \cap (\cup_{k}Hor^{i}_{l_k}))\geq \frac{1}{2}m(Q^{n_k}).$$
Note that $Q^{n_k}$  intersects $Hor^{i}_{l_k}$ and/or $Hor^{i}_{l_k-1}$.
The minimum on the left  of the above inequality occurs, for example,  when $\calz^{n_k}$ falls on the mid-line of the gap between $Hor^{i}_{l_k}$ and $Hor^{i}_{l_k-1}$ because the height of the gap is $4\alpha_{l_{k}-1}$.

 Let $$\widetilde{Q}^{n_k}=\{z\in Q^{n_k}:\ z\in S\subset Q^{n_k} \cap (Hor^{i}_{l_k}\cup Hor^{i}_{l_k-1}) \text{ for  squares } S \in \mathcal{G}^{i}\}.$$
By the above,  $$m(\widetilde{Q}^{n_k})\geq \frac{1}{4}m(Q^{n_k}). $$
 Since  Lemma~\ref{meas square} implies that  for any square $S\subset \cup_{k}^{\infty} Hor^{i}_{l_k}$,  there is a constant $C>0$ such that $m(S\cap Y^\infty)\geq  C m(S)$, it follows that
  $$m(\widetilde{Q}^{n_k}\cap Y)\geq m(\widetilde{Q}^{n_k}\cap Y^\infty)\geq \frac{C}{4}m(Q^{n_k}). $$

Let  $\widetilde{D}_{n_k}=D(\calz^{n_k}, 8\sqrt{2}\alpha_{l_{k-1}})\supset \widetilde{Q}^{n_k}$.  Recall that  $z=F_{n_{k}}(\calz^{n_k})=f^{-n_k} \circ (\calz_{i} )^{-1}(\calz^{n_k})$ and set
$$A=|( \mathcal{Z}_i\circ  f^{n_{k_0}})'(z)| = |(F_{n_{k_0}}^{-1})(\calz^{n_k})'|>0.$$
 
  Let $U_k=F_{n_k}(\widetilde{Q}^{n_k})$ and  denote the respective inscribed and circumscribed circles in $U_k$ by  $D(z,r_k)$ and $D(z,R_k)$.  Since $F_{n_k}$ is univalent on $\widetilde{D}_{n_k}$, it has uniform distortion on $\widetilde{Q}^{n_{k}}$ and by Koebe's theorem there is  a  constant $B>0$ such that
$$
\frac{|F_{n_k}'(\xi)|}{|F_{n_k}'(\eta)|} \leq B, \quad \forall\; \xi, \eta\in D_{n_k}.
$$
Pulling back to the $z$-plane,
 $$
 \frac{R_k}{r_k}\leq B.
 $$

  If $\ell=n_{k}-n_{k_0}$, then  $\Psi^{\ell}=\calz_i \circ f^{\ell} \circ \calz_i^{-1}$ and Proposition~\ref{expanding} implies
  $$|(\Psi^{\ell})'(z)|\geq \frac{ \alpha_{l_k}}{4\pi}.$$  
  Since $$F_{n_k}=(\Psi^{\ell} \circ \mathcal{Z}^{i} \circ f^{n_{k_0}})^{-1},$$  its derivative satisfies 
   $$|F'_{n_k}(z)|\leq \frac{4\pi}{A \alpha_{l_k}}.  $$  
   Because the diameter of $\widetilde{Q}^{n_k}$ is less than $16\sqrt{2}\alpha_{k-1}$,  the diameter of $F_{n_k}(\widetilde{Q}^{n_k})$ tends to $0$, and  therefore $R_{k}\to 0$.

 Furthermore,
 $$
   \frac{m(E \cap D(z, R_k))}{m(D(z,R_k))} \geq \frac{m(E\cap D(z, R_k))}{B^2m(D(z,r_k))} \geq \frac{m(E\cap U_{k}))}{B^2m(U_k))}
  $$
  $$
  \geq \frac{1}{B^2}\frac{m(Y\cap \widetilde{Q}^{n_k})}{B^2m(\widetilde{Q}^{n_k})}\geq \frac{1}{B^4}\frac{m(Y\cap \widetilde{Q}^{n_k})}{m(Q^{n_k})}\geq 
  \frac{C}{4B^{4}}.
   $$
Since this inequality says that
$$
\lim_{k\to \infty} \frac{ m(E\cap D(z, R_k))}{m(D(z,R_k)) }\geq \frac{C}{4B^{4}}>0,
$$
it implies  that the  density of $z$ in  $E$ is positive,   contradicting our assumption that $z$ is a density point of  the complementary set $\CC \setminus E$. 
  Therefore $E$ has full Lebesgue measure.
\end{proof}

With this lemma, we can now prove 

\begin{starthm2}
\label{omegalimset}
For almost every point $z\in \mathbb{C}$, $$\omega(z)=P_f=\cup_{i=1}^N\{\lambda_i, f(\lambda_i),\ldots, f^{p_i-1}(\lambda_i), \infty\}.$$
\end{starthm2}

\begin{proof} Since $f$ is not ergodic,  for almost every point $z\in {\mathbb C}$,
$$
\lim_{n\to \infty} d(f^n(z), P_f)=0.
$$
This implies that $\omega(z)\subseteq P_f$ for almost every point $z\in {\mathbb C}$.  To prove the theorem we need to show that
$
\{z\in {\mathbb C}, \;\;\omega(z)\subsetneq P_f\}
$
has zero Lebesgue measure. 
By Lemma~\ref{C-E meas 0}, we only need to show that
$$
\{z\in E, \;\;\omega(z)\subsetneq P_f\}
$$
has zero Lebesgue measure.

In each plane $\mathcal{Z}_i$,   let $S=S_{\delta_i}$ be a square in $\mathcal{G}_k^i$ in the strip $Hor^i_k$.  
Given $l \in \NN$, and a  $q \in \{1,\ldots, N\}$,   let
 $$K_l^q=\{ \calp^{\boldsymbol{\sigma}}(S)\ :\ \boldsymbol{\sigma}=\{\delta^{i_0},\ldots, \delta^{i_l}\}, i_0\neq m, \ldots, i_l\neq m\}.$$ 
That is,  $K_l^m$ consists of all regions $L$ such that {\em none} of  its successive images under the composition map $\Psi^l$  belong to the $\calz_m$ plane, and whose final image is a square $S'$ in a horizontal strip $Hor^j_{k+l}$ of some $\mathcal{Z}_j-$plane. 

 
  Recall that for any $i$, each square $S' =Rect^i_{k+l,n}\subset Hor^i_{k+l}$ is evenly divided into $N$ rectangles $Rect^i_{j,k+l,n}$; this implies that 
  $$m(S' \cap Rect^i_{j,k+l,n})\leq \frac{1}{N}m(S'),$$ which in turn implies that $$m(S' \cap \cup_{j\neq m}Rect^i_{j,k+l,n})\leq \frac{N-1}{N}m(S').$$


Fix $q'$ and  fix  $L_{q'}=\calp^{\boldsymbol{\sigma}}_{\boldsymbol{\iota}}(S)\in K^l_{q'}$.   The Koebe distortion theorem  for the map $\Psi^l$ implies that there is a distortion factor $D$ such that for each $q$
$$
\frac{m(L_{q'}\cap (\cup_{L\in K_{l+1}^q}L))}{m(L_{q'})}\leq D\frac{m(S'\cap  \cup_{j\neq q}Rect^j)}{m(S')}\leq D\frac{N-1}{N}.
$$
  Because $\Psi^l$ is univalent for all $l$,   its image is outside a large disk in  some $\calz^i$ plane.  Thus, the distortion factor $D$ on each square $S'$ of side $\pi$  is close to $1$.  Thus   we can  assume that $ND/(N-1)<1$;  for example 
  $$D=\frac{N+a}{N-a}, \text{for some } 0<a<\frac{N}{2N-1}$$
Therefore, for each $q$
$$
\frac{m((\cap_{l=0}^{\infty} \cup_{{\mathcal P} \in K_{l}^{q}} {\mathcal P} )\cap S)}{m(S)} =0
$$
and
$$
m((\cap_{n=0}^{\infty} \cup_{{\mathcal P} \in K_{l}^{q}} {\mathcal P} )\cap S) =0.
$$

Finally since $S$ was arbitrary,   set 
$$\mathcal{W}_i=\cup_{q=1}^N\cup_{k}\cup_{S\in Hor_k^i}(\cap_{n=0}^{\infty} \cup_{{\mathcal P} \in K_{l}^{q}} {\mathcal P} ).$$  Thus $m(\mathcal{W}_i)=0$, and if $W=\cup_{i=1}^N Z_i^{-1}(\mathcal{W}_i)$,  then $m(W)=0$; this  implies  $\cup_{n=1}^\infty f^{-n}(W)$ also has zero Lebesgue measure.

To complete the proof, note that $\{z\in E: \omega_f(z)\subsetneq P_f\}\subset \cup_{n=1}^\infty f^{-n}(W)$,  so that it has zero Lebesgue measure. 
\end{proof}

\begin{starthm3}
\label{noabscontmeasure}
 There is no $f$-invariant finite measure   absolutely continuous with respect to Lebesgue measure.
\end{starthm3}

\begin{proof}
Suppose there is an $f-$invariant absolutely continuous measure $\rho$.

Let $\mathcal{W}^i(k)=Y^\infty\cap \{z\in \calz_i, |\Im z|>\alpha_k\}$ and $W_0^i(k)=\calz_i^{-1}(\mathcal{W}^i(k))$.   Then  $$f^{p_i+1}(W_0^i(k))\subset \cup_{i=1}^N W_0^i(k+1) \subsetneq \cup_{i=1}^N W_0^i(k).$$

For each $i=1, \ldots, N$  and $j=1,\ldots, p_i,$ set $W_j^i(k)=f^j(W_0^i(k))$.   For each $k$,  $W_j^i(k)$ is a  bounded set containing $f^j(\la_i)$;  let $\beta_k^j$ be the radius of the maximal disk in  $W_j^i(k)$ centered at  $f^j(\la_i).$ By Lemma~\ref{mapping props},  these disks form a nested sequence whose radii go to zero, so that  as $k \to \infty$, the sequences   $\beta_k^j \to 0$.   
 

Let $\mathbf{W}(k)=\cup_{i=1}^N \cup_{j=0}^{p_i} W_j^i(k)$ and let $p=\max\{ p_1, \ldots, p_N\}$; then, for all pairs $q, n\in \mathbb{N}$,  
 $$f^{(p+1)n+q}(\mathbf{W}(k))\subset \mathbf{W}(k+n). $$

Since $E^\infty\subset \cup W_0^i(1)\subset \mathbf{W}(1)$,  it follows from  Lemma \ref{C-E meas 0}  that
 $$m(\CC \setminus \cup_{n\in \mathbb{N}} f^{-n}(\mathbf{W}(1)))=0.$$ 
 By the  absolutely continuity of  $\rho$, 
 $$\rho(\CC\setminus \cup_{n\in \mathbb{N}} f^{-n}(\mathbf{W}(1)))=0.$$ 
 Moreover, since  $\rho$ is also $f-$invariant,  there is an $r>0$ such that  $\rho(\mathbf{W}(1))=r$.  Furthermore,  since $f^{k(p+1)}(\mathbf{W}(1))\subset \mathbf{W}(k)$ for each $k$, the invariance implies that  $\rho(\mathbf{W}(k)) \geq  \rho(\mathbf{W}(1)) = r$.  

Claim: There exists  an $r'>0$ such that for each $k$, $$\rho(\mathbf{W}(k)\setminus \cup_{i=1}^N W_0^i(k))\geq r'.$$

If not, for all $r'\leq r/2$, there exists a $k$ such that $\rho(\mathbf{W}(k)\setminus \cup_{i=1}^N  W_0^i(k))<r'$.  However, assuming this implies
 $$\rho(\mathbf{W}(k)\setminus \cup_{i=1}^N  W_0^i(k))\geq \rho (\cup_{i=1}^N  W_1^i(k))=\rho(f(\cup_{i=1}^N  W_0^i(k)))\geq \rho(\cup_{i=1}^N W_0^i(k))\geq r-r'  \geq r',$$  which is a contradiction.  Thus the claim holds and so for all $k$,

$$\rho(\cup_i^N\cup_{j=1}^{p_i} W_i^j(k)) \geq r'.$$

Furthermore since  $$\cap_{k=1}^\infty (\cup_{i=1}^N\cup_{j=1}^{p_i} W_i^j(k)))=\cup_{i=1}^N \{\lambda_i, \ldots, f^{p_i-1}(\lambda_i)\},$$ it follows that
$$\rho(\cup_{i=1}^N \{\lambda_i, \ldots, f^{p_i-1}(\lambda_i)\}) \geq r'$$ which contradicts  the absolute continuity of $\rho$.  

\end{proof}

%

%


%
\subsection*{Acknowledgements}
The authors would like to express their sincere gratitude to Weiyuan Qiu for raising the ergodicity question during a talk. The authors are also grateful to the anonymous referees for their careful reading and valuable suggestions, which have significantly improved the manuscript.
\subsection* {Funding} The first author was supported by the AMS–Simons Research Enhancement Grant (No. 01009877) and PSC-CUNY awards. The second author was supported by gifts from the Simons Foundation (Nos. 523341 and 942077) and PSC-CUNY awards. This material is also based upon work supported by the National Science Foundation under Grant No. 1440140, while the third author was in residence at the Mathematical Sciences Research Institute in Berkeley, California, during the Spring 2022 semester.


\begin{thebibliography}{12}


\bibitem[ABF]{ABF} M. Astorg, A. M. Benini and N. Fagella, {\em Bifurcation loci of families of finite type meromorphic maps}. arXiv:2107.02663.


\bibitem[BKL1]{BKL1}  I. N. Baker, J. Kotus and Y. L\"u, {\em Iterates of meromorphic functions II: Examples of wandering domains}. J. London Math. Soc. {\bf 42}(2) (1990), 267-278.
\bibitem[BKL4]{BKL4}  I. N. Baker, J. Kotus and Y. L\"u, {\em Iterates of meromorphic functions IV: Critically finite functions}. Results in Mathematics {\bf 22} (1991), 651-656.
\bibitem[B]{B} W. Bergweiler, {\em Iteration of meromorphic functions}. Bull. Amer. Math. Soc. \textbf{29} (1993), 151-188.
\bibitem[Boc]{Boc} H. Bock, \emph{ On the dynamics of entire functions on the Julia set.} Results. Math. {\bf 30} (1996), 16-20.
\bibitem[C]{C} W. Cui,  \emph{Hausdorff dimension of escaping sets of Nevanlinna functions.} Int. Math. Res. Not. 2021(15) (2021).
\bibitem[CG]{CG} L. Carleson and T. Gamelin, {\em Complex Dynamics}, Springer-Verlag, 1993.
\bibitem[CJK1]{CJK1} T. Chen, Y. Jiang and L. Keen, \emph{Cycle doubling, merging, and renormalization in the tangent family}. Conform. Geom. Dyn. {\bf 22} (2018), 271-314.
\bibitem[CJK2]{CJK2} T. Chen, Y. Jiang and L. Keen, \emph{Accessible boundary points in the shift locus of a family of meromorphic functions with two finite asymptotic values}. Arnold Mathematical Journal volume {\bf 8} (2022), 147–167.
\bibitem[CJK3]{CJK3} T. Chen, Y. Jiang and L. Keen, \emph{Slices of parameter space for meromorphic maps with two asymptotic values}.
Ergodic Theory and Dynamical Systems, Volume {\bf 43}, Issue 1 (2023), 99-139.
\bibitem[CJK4]{CJK4} T. Chen, Y. Jiang and L. Keen,  \emph{Ergodicity in some families of Nevanlinna functions},
Fundamenta Mathematicae 265 (2024), 179-195
\bibitem[CK1]{CK1} T. Chen, L. Keen, \emph{Slices of parameter spaces of generalized Nevanlinna functions}. Discrete and continuous Dynamical Systems, {\bf 39}, Number 10 (2019), 5659-5681.
\bibitem[CK2]{CK2} T. Chen, L. Keen, \emph{Meromorphic functions with a polar asymptotic value}, 	Conform. Geom. Dyn. 28 (2024), 1?36. arXiv:2206.05622.
\bibitem[CW]{CW} W. Cui, J. Wang, \emph{Ergodic exponential maps with escaping singular behaviours}, arXiv:2308.10565
\bibitem[Dev]{Dev} R. Devaney, \emph{ Structurally Instablity of $Exp(z)$,} Proceedings of the American Mathematical Society, 94(1985), 545-548.
\bibitem[DK]{DK} R. Devaney, L. Keen, \emph{Dynamics of meromorphic functions: functions with polynomial Schwarzian derivative,} Ann. Sci. \'Ec. Norm. Sup\'er. (4) {\bf 22} (1989), 55-79.
\bibitem[EM]{EM} A. Eremenko and S. Merenkov, \emph{Nevanlinna functions with real zeros}, Ill. J. Math., Vol. 49 {\bf 4}, (2005) 1093-1110.
\bibitem[FK]{FK} N. Fagella, L. Keen,  \emph{Stable components in the parameter plane of transcendental functions of finite type},  J Geom Anal. {\bf 31} (2021), 4816-4855 .
\bibitem[GGS]{GGS} E. Ghys, L. Goldberg and D. Sullivan, \emph{On the measurable dynamics of $z\to e^z$, } Ergodic Theory Dynamical Systems, 5, No.3, (1985), 329-335.
\bibitem[H]{H}  E. Hille,\emph{ Ordinary Differential Equations in the Complex Domain.} Wiley, New York (1976).
\bibitem[KK1]{KK1} L. Keen, J. Kotus, \emph{Dynamics of the family of $\lambda\tan z$}, Conform. Geom. Dyn. Vol. {\bf 1} (1997), 28-57 .
\bibitem[KK2]{KK2} L. Keen, J. Kotus, \emph{Ergodicity of some family of meromorphic functions}, Ann. Acad. Sci. Fenn. {\bf 24} (1999),133-145.
\bibitem[L]{L}J. K. Langley, \emph{Postgraduate notes on complex analysis.}  Preprint.
\bibitem[Lyu1]{Lyu1} M. Lyubich, \emph{ On typical behaviour of the trajectories of a rational mapping on the sphere.} Soviet Math. Dokl. 27 (1983), 22-25.
\bibitem[Lyu2]{Lyu2} M. Lyubich, \emph{Measurable dynamics of the exponential.} Sib. Math. J. {\bf 28}(1987), 780-793.
\bibitem[Mc]{McM} C. McMullen, \emph{Complex Dynamics and Renormalization}, Annals of Math Studies, Vol. {\bf 135}, Princeton Univ. Press, Princeton, NJ, 1994.
\bibitem[M]{M} J. Milnor,   {\em Dynamics in one complex variable}. Third edition. Annals of Mathematics Studies, 160. Princeton University Press, Princeton, NJ, 2006.

\bibitem[Mis]{Mis} M.  Misiurewicz, \emph{On iterates of $e^z$},  Ergodic Theory and Dynamical Systems. 1981;1(1):103-106.
\bibitem[N]{N} R. Nevanlinna, \emph{Analytic functions.} Die Grundlehren der mathematischen Wissenschaften, Band 162, Springer-Verlag, New York-Berlin, 1970. Translated from the second German edition by Phillip Emig.
\bibitem[S]{Skor} B. Skorulski, \emph{Non-ergodic maps in the tangent family}. Indagationes Mathematicae, Vol. {\bf 14} (1) (2003), 103-118.
\bibitem[S1]{Skor1}B. Skorulski, \emph{Metric properties of the Julia set of some meromorphic functions with an asymptotic value eventually mapped onto a pole}. Math. Proc. Cambridge Philos. Soc. 139 (2005), 117-138.
\bibitem[Su] {Sul} D. Sullivan,  {\em Quasiconformal homeomorphisms and dynamics. I. Solution of the Fatou-Julia problem on wandering domains.} Ann. of Math. (2) {\bf 122} (1985), no. 3, 401-418.

\bibitem[WZ]{WZ} X. Wang, G. Zhang, \emph{Constructing ergodic exponential maps with dense post-singular orbits}, Ergod. Th. and Dynam. Sys. (2010), 30, 309-316.

\bibitem[Z]{Z} S. Zakeri, {\em A Course in Complex Analysis}, Princeton University Press, Princeton, NJ, 2021.
\end{thebibliography}
\end{document}